\documentclass[10pt]{article}
\pdfoutput=1

\ifdefined\texmaker
\usepackage[margin=1.5cm,paperwidth=6in]{geometry}
\else
\usepackage[margin=1.5in]{geometry}
\fi

\usepackage{amsmath}
\usepackage{amsthm}
\usepackage{amsfonts}
\usepackage{amssymb}
\usepackage[shortalphabetic,initials]{amsrefs}
\usepackage{amscd}	
\usepackage{cancel}	

\usepackage{graphicx}

\usepackage{ifpdf}
\usepackage{color}

\ifdefined\draftversion
\providecommand{\showkeys}{}
\providecommand{\gitaddpagehash}{}
\fi

\ifdefined\showkeys
	\usepackage[color]{showkeys} 
\fi

\ifpdf
\usepackage{hyperref}
\fi

\makeatletter
\def\blfootnote{\xdef\@thefnmark{}\@footnotetext} 
\makeatother

\IfFileExists{.git/git.tex}{
\input{.git/git.tex}
\newcommand{\gitversionfootnote}{\blfootnote{Version: \gitauthsdate, \gitshash}}

\ifdefined\gitaddpagehash
  \usepackage{fancyhdr}
  \pagestyle{fancy}
  
  \lhead{} \rhead{}
  \rfoot{\tiny\gitshash}
\fi

}
{
\newcommand{\gitversionfootnote}{}
}

\graphicspath{{}}
\ifpdf
\newcommand{\fig}[2]{\def\svgwidth{#2}\input{#1.pdf_tex}}
\else
\newcommand{\fig}[2]{\frame{PDF Figure here}}	
\fi

\newcommand{\dimension}{n}

\newcommand{\dist}{\operatorname{dist}}
\newcommand{\diam}{\operatorname{diam}}

\newcommand{\interior}{\operatorname{int}}
\newcommand{\trace}{\operatorname{tr}}

\newcommand{\esssup}{\operatorname{ess\ sup}\displaylimits}
\newcommand{\essinf}{\operatorname{ess\ inf}\displaylimits}

\newcommand{\lap}{\Delta}

\newcommand{\ov}[1]{\frac{1}{#1}}

\newcommand{\abs}[1]{\left|#1\right|}
\newcommand{\pth}[1]{\left(#1\right)}
\newcommand{\bra}[1]{\left[#1\right]}
\newcommand{\set}[1]{{\left\{#1\right\}}}
\newcommand{\at}[2]{{{\left.{#1}\right|}_{#2}}}

\newcommand{\cl}[1]{\overline{#1}}	

\newcommand{\al}{\ensuremath{\alpha}}
\newcommand{\be}{\ensuremath{\beta}}
\newcommand{\ga}{\ensuremath{\gamma}}
\newcommand{\de}{\ensuremath{\delta}}
\newcommand{\e}{\ensuremath{\varepsilon}}
\newcommand{\vp}{\ensuremath{\varphi}}
\newcommand{\la}{\ensuremath{\lambda}}
\newcommand{\si}{\ensuremath{\sigma}}
\newcommand{\ta}{\ensuremath{\theta}}

\newcommand{\om}{\ensuremath{\omega}}

\newcommand{\R}{\ensuremath{\mathbb{R}}}

\newcommand{\Rd}{\ensuremath{{\mathbb{R}^{\dimension}}}}
\newcommand{\Rn}{\Rd}

\newcommand{\N}{\ensuremath{\mathbb{N}}}

\newcommand{\limhalfsup}{\sideset{}{^*}\limsup}
\newcommand{\limhalfinf}{\sideset{}{^*}\liminf}

\definecolor{grey}{rgb}{0.6,0.6,0.6}

\renewcommand{\labelenumi}{(\alph{enumi})}

\newcommand{\romanlist}{\renewcommand{\labelenumi}{\textup{(}\roman{enumi}\textup{)}}}

\numberwithin{equation}{section}

\swapnumbers	
\newtheorem{theorem}{Theorem}[section]
\newtheorem{lemma}[theorem]{Lemma}
\newtheorem{proposition}[theorem]{Proposition}
\newtheorem{corollary}[theorem]{Corollary}

\newtheorem{definition}[theorem]{Definition}

\theoremstyle{definition}

\newtheoremstyle{remarkstyle}
  {3pt}
  {3pt}
  {\small}
  {}
  {\bfseries}
  {.}
  { }
  {}

\theoremstyle{remarkstyle}
\newtheorem{remark}[theorem]{Remark}
\newtheorem{example}[theorem]{Example}

\title{Nonlinear elliptic-parabolic problems}

\author{Inwon C. Kim\thanks{Dept. of Mathematics, UCLA, USA. Partially supported by NSF DMS-0970072} and Norbert Po\v{z}\'{a}r\thanks{Dept. of Mathematics, University of Tokyo, Japan.}}\date{}

\begin{document}
\maketitle
\begin{abstract}
We introduce a notion of viscosity solutions
for a general class of elliptic-parabolic phase transition
problems. These include the Richards equation, which is
a classical model in filtration theory. Existence and
uniqueness results are proved via the comparison
principle. In particular, we show existence and stability
properties of maximal and minimal viscosity solutions for
a general class of initial data. These results are new
even in the linear case, where we also show that viscosity
solutions coincide with the regular weak solutions
introduced in \cite{AL}.
\end{abstract}
\gitversionfootnote

\section{Introduction}

Let $\Omega \subset \Rn$ be a smooth bounded domain and $T > 0$. Let us denote $Q = \Omega \times (0,T]$.
We are interested in the following problem: find a function $u(x,t)$, $u : \cl Q \to \R$, that solves 
\begin{align}
\label{RP}
\begin{cases}
\partial_t b(u) - F(D^2 u, Du, u) = 0 & \text{in } Q, \\
u = g \equiv -1 & \text{on } \partial \Omega \times [0,T],\\
u(\cdot,0) = u_0 & \text{on } \Omega,
\end{cases}
\end{align}
where $Du$ denotes the spatial gradient of $u$, $D^2u$ is the spatial Hessian, and $F(M,p,z):\mathcal{S}_n\times \R^n\times\R \to \R$ is a fully nonlinear, uniformly elliptic operator (see Section~\ref{sec:assumptionsOnF} for precise assumptions on $F$). For the function $b:\R\to\R$ we assume that 
\begin{enumerate}
\item $b$ is increasing and Lipschitz,
\item $b(s) = 0$ for $b \leq 0$, $b \in C(\R) \cap C^1([0,\infty))$, 
\item there exists a constant $c  >0$ such that $b'(s) > c$ for $s \in (0,\infty)$. 
\end{enumerate}

The nonlinear operators $F$ we consider include:
\begin{itemize}
\item 
the uniformly elliptic operator of non-divergence form
\begin{equation}\label{nondivergence}
F(D^2u, Du,u) = -\trace (A(Du)D^2u) + H(u,Du),
\end{equation}

where $A$ satisfies the uniform ellipticity condition
\begin{equation}\label{elliptic1}
\lambda |q|^2 \leq A(p)q\cdot q \leq \Lambda |q|^2 \hbox{ for all } p,q\in\R^n, \hbox{ for some } \lambda, \Lambda>0,
\end{equation}
and
as well as the Bellman-Issacs operators arising from stochatic optimal control and differential games
$$
F(D^2u, Du,u) = \inf_{\alpha\in A} \sup_{\beta\in B} \{\mathcal{L}^{\alpha\beta}u \},
$$
where $\mathcal{L}^{\alpha \beta}$ is a two-parameter family of operators of the form \eqref{nondivergence} satisfying \eqref{elliptic1}; we refer to \cites{CIL, FS} for further examples. 
\item 
a divergence-form operator; to simplify our discussion, we restrict our attention to operators of the form
\begin{equation}\label{divergence}
F (D^2u, Du, u)= \nabla\cdot (\Psi(b(u)) Du),
\end{equation}
where $\Psi\in C^1([0,\infty))$ is a positive function.
The class of operators given in \eqref{divergence} is of particular interest since in that case the problem \eqref{RP} represents the well-known \emph{Richards equation}, which serves as a basic model for the filtration of water in unsaturated soils (see e.g. \cites{DS, MR, R}).  
\end{itemize}

\vspace{10pt}
Our aim is to study the well-posedness of \eqref{RP}. Note that, due to the regularity theory for uniformly elliptic nonlinear operators (\cites{CC, W}), solutions of \eqref{RP} satisfy interior $C^{1,\alpha}$ estimates in the sets $\{u>0\}$ and $\{u<0\}$. Hence the challenge in the study of the problem lies in the behavior of a solution near the transition boundary between the positive and the negative phases: for example, as illustrated in Example~\ref{ex:jump}, discontinuities of solutions in time across the set $\{u=0\}$ are generic.

\medskip

Problem \eqref{RP} can be understood as the limiting equation for the evolution of two phases with different time scales of diffusion and with the threshold value at $u=0$.
In particular, as verified in section~\ref{sec:approximation}, the problem can be viewed as a singular limit of a family of  uniformly parabolic problems \eqref{approx}, where $b(s)$ in $\eqref{RP}$ is regularized.
Hence it is expected that a maximum principle holds for the solutions, and the theory of viscosity solutions may be applicable for the study of pointwise behavior of solutions near the transition boundary $\{u=0\}$.
This is indeed our approach: in this paper we will introduce the notion of viscosity solutions for \eqref{RP} and discuss existence, uniqueness and stability properties, and compare them to the notion of weak solutions (see the discussion below).
Such results have been established for Stefan-type problems (see \cites{CS, KP}, for example), but significant challenges in the analysis arise due to the implicit nature of the boundary motion law in \eqref{RP} and the generic nature of discontinuities in $u$ (as in Example~\ref{ex:jump}); see an extended discussion on this in section~\ref{sec:comparison}.
We aim to present the proof of the comparison principle for fully nonlinear operators in more detail, both to illustrate the flexibility of the viscosity solution approach and to make the results readily available for applications in a general context.
 
\medskip
 
 It should be pointed out that our approach treats the transition boundary $\Gamma:=\partial\{u>0\}$ between the elliptic and parabolic regions as a ``free boundary'' and constructs barriers based on its movement. Hence our approach may not be optimal for other problems such as fast diffusion (when $b'(0)=0$) where the transition boundary moves with infinite speed (see \cite{V}, for example), but strong regularity properties for the solutions are expected (see \cite{BH}).

\vspace{10pt}
Although our presentation is mainly focused on the problem with fully nonlinear operators whose structure assumption requires a linear growth in terms of $|Du|$, the results of this paper extend to the case of divergence-form operators of the form \eqref{divergence} as well. We point out the necessary modifications in the text where appropriate.

\subsubsection*{Literature review}

Let us briefly discuss previously known results on \eqref{RP}: all of them concern divergence-form operators including \eqref{divergence}.
The weak solutions are defined via integration by parts in the important paper of Alt \& Luckhaus \cite{AL}, which shows existence (\cite{AL}*{Theorem 1.7}) of weak solutions for the general class of elliptic-parabolic phase transition problems with divergence-form operators.
Uniqueness results are, however, rather limited.
For $F$ given by \eqref{divergence}, when $\Psi(s)$ is a positive constant (that is, when $F$ is linear), the authors  prove that regular weak solutions (in the sense $\partial_t b(u)\in L^2$) can be constructed by the Galerkin method (\cite{AL}*{Theorem 2.3}), and show that regular weak solutions satisfy the comparison principle and are unique given the initial data (\cite{AL}*{Theorem 2.2}).
Indeed, when $\Psi(s)$ is a positive constant, it is known that $b(u)$ is continuous in a local setting as long as $u$ is bounded (see \cite{DiBG}).
Continuity of $b(u)$ seems to be the optimal result for this problem with general initial data, since $u$ may become discontinuous in time in the elliptic phase (see Example 1.2, and also \cite{DiBG}).
The proof of continuity in \cite{DiBG} is based on the weak Harnack inequality, making use of the linearity of the elliptic operator with respect to $u$.
We also refer to \cite{BW} and \cite{Carrillo99}, who use an entropy solution approach to define weak solutions as well as to prove comparison principle in $L^1$ for the relevant model; this approach, while powerful, does not fit into our setting where we have a non-vanishing elliptic phase.

\medskip

Even for the quasi-linear $F$ given in \eqref{divergence}, there are no uniqueness or stability results except for the aforementioned linear case and for the one-dimensional case (see \cites{BH, VP}): this serves as a motivation of our analysis in this paper.
In one dimension, Mannucci \& Vazquez \cite{MV} studied viscosity solutions of \eqref{RP} for divergence-form operators. Their approach avoids possible complications at the transition boundary $\partial \set{u>0}$ by relying on previously known regularity properties of weak solutions in one dimension.

\subsubsection*{Summary of the main results}

In this section we summarize the main results obtained in this paper. In all statements $F$ is assumed to be either a fully nonlinear operator satisfying the assumptions in Section~\ref{sec:assumptionsOnF} or a quasilinear divergence-form operator of the form \eqref{divergence}.

Our main theorem is the following comparison principle: 

\swapnumbers
\theoremstyle{plain}
\newtheorem{thm}{Theorem}
\begin{thm}[Theorem~\ref{th:comparisonPrinciple} and Theorem~\ref{CP:div}]\label{main1}
Let $u$ be a viscosity subsolution and $v$ a viscosity supersolution of \eqref{RP} on $Q = \Omega \times (0, T]$ for some $T > 0$. If  $u < v$ on the parabolic boundary $\partial_P Q$, then $u < v$ on $Q$.
\end{thm}

Equipped with the comparison principle, we use Perron's method to show the following existence and stability theorem.
\begin{thm}\label{main2}
For initial data  $u_0 \in \mathcal{P}$ (see the definition of $\mathcal{P}$  in section~\ref{sec:existenceAndStability}, by \eqref{A}--\eqref{B}), the following holds:
\begin{enumerate}
\item (Theorem~\ref{existence}) There exists a minimal and a maximal viscosity solution  $\underline{u}$ and $\overline{u}$ of \eqref{RP} with initial data $u_0$.
\item (Theorem~\ref{stability:visc}) $\underline{u}$ and $\overline{u}$ are stable under perturbations of initial data with appropriate ordering.
\item (Corollary~\ref{approx11}) $\underline{u}$ and $\overline{u}$ can be obtained as a limit of solutions solving the regularized parabolic equation \eqref{approx}.
\end{enumerate}
\end{thm}

Aforementioned theorem states that the maximal and minimal viscosity solutions are stable. Unfortunately, we are only able to show the uniqueness of general viscosity solutions (i.e. the coincidence of minimal and maximal viscosity solutions) in several restricted settings. The coincidence of the minimal and the maximal viscosity solutions with general initial data remains open, except for the linear case.

\begin{thm}\label{main3}
For given initial data $u_0\in\mathcal{P}$, the following holds:
\begin{enumerate}
\item (Theorem~\ref{coincidence}) If $F$ is linear, i.e. if $F(M,p,z)=F(M) = \trace M$, then there exists a unique viscosity solution $u$ with initial data $u_0$, and $u$ coincides with the unique weak solution defined in \cite{AL}.
\item (Theorem~\ref{uniqueness}) If $u_0$ is either star-shaped or if $u$ decreases at $t=0$, then there exists a unique viscosity solution of \eqref{RP}.
\end{enumerate}
\end{thm}

\begin{remark}
As mentioned above, our approach may not be optimal if $b(s)$ degenerates at $s=0$.
On the other hand, we expect that our approach can be extended to non-Lipschitz $b(s)$ and produce results similar to the above. The difficulty in the analysis lies in the corresponding degeneracy of the elliptic operator in the positive phase, when we write the equation in terms of $b(s)$. 
\end{remark}

\begin{example}[Discontinuous solution]
\label{ex:jump}
Here we briefly discuss an example which illustrates discontinuities in the solutions. 
Set 
\begin{align*}
b(u) = u_+:=\max (0, u),
\end{align*}
and consider \eqref{RP} with negative boundary data, and initial data that are positive on some open set. As the solution evolves, the positive phase disappears in finite time, and then the solution jumps to the stationary solution. Nevertheless, one expects $b(u)$ to be continuous. We refer the reader to \cite{AL}*{p. 312} for an explicit formula.
\end{example}

\subsection{Assumptions on the nonlinear operator \texorpdfstring{$F$}{F}}
\label{sec:assumptionsOnF}

Let $\mathcal{S}_n$ be the space of symmetric $n \times n$ matrices.
For given $0 < \la \leq \Lambda$, we define the \emph{Pucci extremal operators} $\mathcal{M}^\pm : \mathcal{S}_n \to \R$ as in \cites{CC, WangI}:
\begin{align}
\label{eq:Pucci}
\mathcal{M}^+(M) &= \sup_{A \in [\la I, \Lambda I]} \trace AM, & 
\mathcal{M}^-(M) &= \inf_{A \in [\la I, \Lambda I]} \trace AM,
\end{align}
where $[\la I, \Lambda I] = \set{A \in \mathcal{S}_n : \lambda I \leq A \leq \Lambda I}$.
Alternatively, the Pucci operators can be expressed using the eigenvalues $e_1, \ldots, e_n$ of matrix $M$:
\begin{align*}
\mathcal{M}^+(M) &= \Lambda \sum_{e_i > 0} e_i + \lambda \sum_{e_i < 0} e_i, & 
\mathcal{M}^-(M) &= \lambda \sum_{e_i > 0} e_i + \Lambda \sum_{e_i < 0} e_i.
\end{align*}

With the Pucci operators at hand, we shall assume the following structural condition on the operator 
$$
F(M,p,z): \mathcal{S}_n\times \R^n\times\R \to \R:
$$

\begin{enumerate}
\romanlist
\item  There exist $0 < \la < \Lambda$ and $\de_0, \de_1 \geq 0$ such that
\begin{align}
\label{eq:structural}
\begin{aligned}
\mathcal{M}^-(M-N) - \de_1 \abs{p - q} &- \de_0 \abs{z - w} \leq F(M, p, z) - F(N, q, w) \\
&\leq \mathcal{M}^+(M-N) + \de_1 \abs{p - q} + \de_0 \abs{z - w}
\end{aligned}
\end{align}
for all $M, N \in \mathcal{S}_n$, $p, q \in \Rn$ and $z, w \in \R$.

\item $F$ is proper, i.e.
\begin{align}
\label{eq:proper}
z \mapsto F(M, p, z) \quad \text{is nonincreasing in $z$}. 
\end{align}

\item Finally, to guarantee that $u \equiv 0$ is a solution of both the parabolic and the elliptic problems, we assume that
\begin{align}
\label{eq:Fzero}
F(0,0,0) = 0.
\end{align}
\end{enumerate}

\subsection{Notation}

\newcommand{\notationbullet}{\medskip\noindent$\bullet$ }

\notationbullet
In this paper, we work in a fixed space dimension $n \geq 2$. For a point $x \in \Rn$ and time $t \in \R$, the pair $(x,t) \in \Rn \times \R$ represents a point in space-time.

\notationbullet
For given $r > 0$, we define the open balls
\begin{align*}
B_r &:= \set{(x,t): \abs{x}^2 + \abs{t}^2 < r^2},&
B_r^n &:= \set{x : \abs{x} < r},
\end{align*}
the space disk
\begin{align*}
D_r := B_r^n \times \set{0} = \set{(x, 0) : \abs{x} < r},
\end{align*}
and the flattened set
\begin{align*}
E_r  := \set{(x,t) : \abs{x}^3 + \abs{t}^2 < r^2}.
\end{align*}
Finally, we define the domain that is used in the definition of regularizations of solutions,
\begin{align*}
\Xi_r := D_r + E_r.
\end{align*}
Here $+$ is the Minkowski sum. Note that $\Xi_r \in C^2$ (in contrast with $D_r + B_r$, which is only $C^{1,1}$).

\notationbullet
We ask the reader to forgive a slight abuse of notation:
\begin{align*}
\partial D_r &:= \cl{D_r} - D_r = \partial B_r^n \times \set{0} = \set{(x,0) : \abs{x} =  r}.
\end{align*}

\notationbullet
It will also be advantageous to introduce the (open) top and bottom flat pieces of $\partial \Xi_r$,
\begin{align*}
\partial^\top \Xi_r &:= \set{(x,r) : \abs{x} < r}, & \partial_\bot \Xi_r &= \set{(x,-r) : \abs{x} < r},
\end{align*}
and the (open) lateral boundary of $\Xi_r$,
\begin{align*}
\partial_L \Xi_r := \partial \Xi_r \setminus \cl{\partial^\top \Xi_r \cup \partial_\bot \Xi_r}.
\end{align*}
These sets are sketched in Figure~\ref{fig:Xi}.
\begin{figure}
\centering
\fig{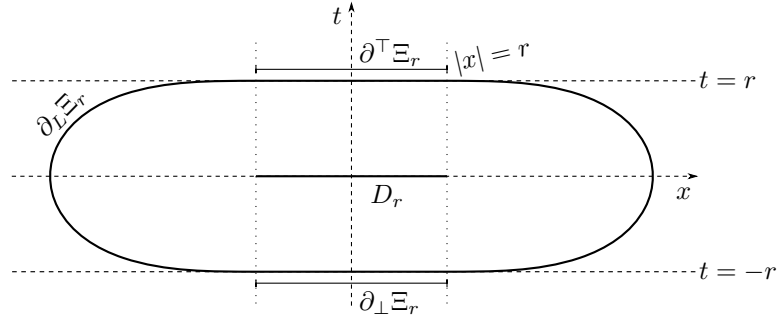}{4.5in}
\caption{The boundary $\partial \Xi_r$ of the set $\Xi_r$}
\label{fig:Xi}
\end{figure}

\notationbullet
The translation of a set $A \in \Rn \times \R$ by a vector $(x,t) \in \Rn \times \R$ will be denoted as
\begin{align*}
A(x,t) := (x,t) + A.
\end{align*}
The translation $A(x)$ of a set $A \subset \Rn$ is defined similarly.

\notationbullet
We will often need to consider timeslices (cross sections at a fixed time) of a given set. To simplify the notation, let us define the timeslice of a set $A \subset \Rn \times \R$ at time $t$ as
\begin{align*}
\at{A}{t} := \set{x : (x, t) \in A}.
\end{align*}
We often write $A_t$ if there is no ambiguity.

\notationbullet
Let $E \subset \Rn\times \R$. Then $USC(E)$ and $LSC(E)$ are respectively the sets of all upper semi-continuous and lower semi-continuous functions on $E$.
For a locally bounded function $u$ on $E$ we define the semi-continuous envelopes
\begin{align}\label{usclsc}
u^{*,E} &:= \inf_{\substack{v \in USC(E)\\v \geq u}} v, & u_{*, E} &:= \sup_{\substack{v \in LSC(E)\\v \leq u}} v.
\end{align}
These envelopes are used throughout most of the article with $E = \cl Q$, and therefore we simply write $u^*$ and $u_*$ if the set $E$ is understood from the context.

\section{Viscosity solutions}
\label{sec:viscositySolutions}

In this section we define the notion of viscosity solutions of problem \eqref{RP}.
Formally, viscosity solutions are the functions that satisfy a local comparison principle on parabolic neighborhoods with barriers which are the classical solutions of the problem.
We refer the reader to \cites{CIL, CV, CS, KP} and references therein for other examples of this approach.

\begin{definition}[Parabolic neighborhood and boundary]
A nonempty set $E \subset \Rn\times \R$ is called a \emph{parabolic neighborhood} if $E = U \cap \set{t \leq \tau}$ for some open set $U \subset \Rn\times \R$ and some $\tau \in \R$. Let us denote $\partial_P E := \cl{E} \setminus E$, the \emph{parabolic boundary} of $E$. (See \cite{WangI} for a more general definition.)
\end{definition}

\begin{definition}[Classical subsolution]
\label{de:classicalSubsolution}
Let $E$ be a parabolic neighborhood. Function $\vp$ is called a \emph{classical subsolution} of problem \eqref{RP} in a parabolic neighborhood $E$ if $\vp \in C(\cl{U})$ on an open set $U \in \Rn \times \R$ such that $E = U \cap \set{t \leq \tau}$ for some $\tau \in \R$, and the following holds:
\begin{enumerate}
\romanlist
\item $\vp \in C^{2,1}_{x,t}(\cl{\set{\vp > 0}})$ and $C^{2,1}_{x,t}(\cl{\set{\vp < 0}})$,
\item $\set{\vp = 0} \subset \partial\set{\vp > 0} \cap \partial\set{\vp < 0}$ and $\abs{D\vp^\pm}  > 0$ on $\set{\vp = 0}$,
\item $b(\vp)_t - F(D^2 \vp, D\vp, \vp) \leq 0$ on $\set{\vp > 0}$ and $\set{\vp < 0}$,
\item $\abs{D\vp^+} \geq \abs{D\vp^-}$ on $\set{\vp = 0}$.
\end{enumerate}
Here $\set{\vp > 0} := \set{(x,t) \in U : \vp(x,t) > 0}$ etc., and
$$
D \vp^\pm(\xi,\tau) := \lim_{\substack{(x,t) \to (\xi,\tau)\\(x,t) \in \set{\pm \vp > 0}}} D\vp(x,t).
$$ 

We say that $\vp$ is a \emph{strict} classical subsolution if the inequalities in (iii) and (iv) are strict.
\end{definition}

Classical supersolutions are defined similarly by flipping the inequalities in Definition~\ref{de:classicalSubsolution} (iii)--(iv).

At last, we define viscosity solutions.
Note that we set $g \equiv -1$ on $\partial \Omega \times [0,T]$ throughout the paper.

\begin{definition}[Viscosity subsolution]
\label{de:viscositySubsolution}

Let $Q = \Omega \times (0,T]$ be a parabolic cylinder. Function $u \in USC(\cl Q)$ is a \emph{viscosity subsolution} of \eqref{RP} in $Q$ if $u(\cdot, 0) \leq u_0$ on $\Omega$, $u \leq g$ on $\partial \Omega \times [0,T]$, and if $u < \vp$ on $E$ for any \emph{strict} classical supersolution $\vp$ on any parabolic neighborhood $E \subset Q$ for which $u < \vp$ on $\partial_P E$.
\end{definition}

One can define viscosity supersolutions accordingly, as a function in $LSC(\cl{Q})$, by switching the direction of inequality signs in the previous definition. 

\begin{definition}[Viscosity solution]
Locally bounded function $u$ is a viscosity solution of \eqref{RP} on $Q$ if $u^{*,\cl Q}$ is a viscosity subsolution on $Q$ and $u_{*, \cl Q}$ is a viscosity supersolution on $Q$.
\end{definition}

\begin{remark}
\label{re:only-cylinders}
We only test viscosity solutions by \emph{strict} classical barriers.
It is therefore possible to narrow the choice of $E$ in Definition~\ref{de:viscositySubsolution} to only include parabolic cylinders of the form $Q' = \Omega' \times (t_1, t_2] \subset Q$, where $\Omega'$ has a smooth boundary, instead of all parabolic neightborhoods.
Indeed, suppose that $E \subset Q$ is a parabolic neighborhood, $\vp$ is a strict classical supersolution on $E$, $u < \vp$ on $\partial_P E$, but $u \geq \vp$ at some point in $E$. Define $\tau := \sup \set{\si: u < \vp \text{ on } E \cap \set{t \leq \si}} \in \R$. The set 
\begin{align*}
A := \set{x : (x,\tau) \in E, u \geq \vp}
\end{align*}
is compact and therefore $\de := \dist(A \times \set{\tau}, \partial_P E) > 0$. Define the parabolic cylinder $Q' = (A + B_{\de/2}) \times (\tau - \de/2, \tau]$. Clearly $Q' \subset E$ and $u < \vp$ on $\partial_P Q'$. The boundary of $A + B_{\de/2}$ can be easily regularized. This observation will be useful in section~\ref{sec:uniqueness}, where we will show that regular weak solutions are viscosity solutions.
\end{remark}

\section{Comparison principle} 
\label{sec:comparison}

This section is devoted to the proof of the following ``weak'' comparison principle.

\begin{theorem}
\label{th:comparisonPrinciple}
Let $u$ be a viscosity subsolution and $v$ a viscosity supersolution of \eqref{RP} on $Q = \Omega \times (0, T]$ for some $T > 0$, and assume that $u < v$ on $\partial_P Q$. Then $u < v$ on $Q$.
\end{theorem}

To simplify the exposition of the proof of this theorem, we shall assume that $T = \infty$. 
In fact, it is always possible to extend $u$ and $v$ from $\Omega \times (0,T]$ to $\Omega \times (0, \infty)$.
Moreover, we will only consider $b(u)$ of the form $b(u) = u_+ := \max(u, 0)$.
The problem \eqref{RP} with a fully nonlinear operator $F$ and a more general $b$ can be always rewritten in this way.
Indeed, the differentiation $b(u)_t = b'(u) u_t$ for $u > 0$ is justified by the regularity of $b$, and $b'(u)$ can be absorbed into $F$.

\subsubsection*{Heuristic arguments}

The comparison principle for classical subsolutions and supersolutions $u$ and $v$ of \eqref{RP} can be proved using the following formal argument:
we would like to show that 
\begin{equation}\label{order}
\{u(\cdot,t)>0\}\subset \{v (\cdot,t)>0\}\hbox{ for all }t>0,
\end{equation}
 since then the conclusion follows due to the standard elliptic and parabolic comparison principle.
Hence suppose \eqref{order} fails at some time. Since $|Du|, |Dv|>0$ on the boundary of their respective positive phases, the sets $\{u(\cdot,t)>0\}$ and $\{v(\cdot,t)>0\}$ have smooth boundaries and evolve continuously in time with respect to the Hausdorff distance. Therefore it follows that there exists the first time $t_0>0$ where $\partial\{u(\cdot,t)>0\}$ intersects $\partial\{v(\cdot,t)>0\}$, let's say at $x=x_0$. Since 
 $$
 \{u>0\}\cap\{t\leq t_0\} \subset \{v>0\}\cap\{t\leq t_0\} \hbox{ and } u_0 <v_0,
 $$
the comparison principle applied to $u$ and $v$ respectively in the sets $\{u>0\}$ and in $\{v<0\}$  yields that $u\leq v$ up to $t=t_0$. In particular,  since $u(x_0,t_0)=v(x_0,t_0)=0$, this means that 
 $$
 |Du^+|(x_0,t_0) \leq |Dv^+|(x_0,t_0) \hbox{ and } |Dv^-|(x_0,t_0) \leq |Du^-(x_0,t_0)|.
 $$
  Furthermore, due to the regularity of $\partial\{u(\cdot,t>0\}$ and $\partial\{v(\cdot,t)>0\}$ and Hopf's lemma for uniformly parabolic  and elliptic operators,  it turns out that the above inequalities are in fact strict. This contradicts the flux-matching condition for the classical sub- and supersolution (Definition~\ref{de:classicalSubsolution}(iv)), i.e. the fact that
  $$|Du^+| \geq|Du^-|\hbox{ and }|Dv^+|\leq|Dv^-|\hbox{ at } (x_0,t_0). $$ 

\vspace{10pt}

Unfortunately, the rigorous version of the above heuristic argument is rather lengthy. There are many difficulties one faces in the general setting, where $u$ and $v$ are merely semi-continuous functions.  As it is always the case in the theory of weak solutions, one should translate the above heuristics onto appropriate test functions or, to be more precise in our case, barriers. Our argument relies on a certain regularization procedure (see subsection~\ref{sec:regularizations}) which ensures that, at a contact point $(x_0,t_0)$ of the regularized solutions, the phase boundaries of each regularized solution are both locally $C^{1,1}$ in space. Such regularity of the phase boundary would enable us to construct appropriate barriers which would allow one mimic the heuristic argument above. This technique was pioneered by Caffarelli and V\'{a}zquez (\cite{CV}) in their treatment of viscosity solutions for the porous medium equation. It was later applied to several one phase free boundary problems in \cites{BV, K03, K02}, and later extended to two-phase Stefan problems by the authors in \cite{KP} (see also \cite{CS}). 

\vspace{10pt}

In contrast to the aforementioned results for free boundary problems, the analysis of our problem presents several new challenges. The most obvious challenge arises from the flux matching condition of \eqref{RP} on the transition boundary. While the regularization procedure provides regularity information in space variables, one should still show the finite propagation property of the phase boundary. In the aforementioned free boundary problems, the free boundary motion law prescribes the normal velocity of the free boundary in terms of the gradient of the solution, which links the space regularity to the time regularity of the solution. Here one does not have such a direct relationship between time and space regularity of the transition boundary. Indeed, the flux matching condition of \eqref{RP} turns out to be more delicate than the prescribed gradient condition in flame propagation type problems, since one has to account for the possibility that the fluxes from both sides will either degenerate to zero or diverge to infinity.  This is overcome by the observation that with the regularization we can rule out the scenarios of a sudden extinction of the elliptic phase (Lemma~\ref{le:trunk}), or a sudden shrinkage/discontinuous expansion of the parabolic phase (Lemmas~\ref{le:DsupersolutionInfty} and \ref{le:finiteSpeed}) at a contact point. Note, however, that even with regularized solutions, the elliptic phase might instantly become extinct away from a contact point.

\vspace{10pt}

 We point out that, to allow for the regularization procedure, the strict ordering of $u$ and $v$ on $\partial_P Q$ in the statement of Theorem~\ref{th:comparisonPrinciple} is necessary. We also point out that a proof via doubling of variables, a classical tool in the theory of viscosity solutions (see \cite{CIL}), is not available for phase transition-type problems, including \eqref{RP}.

\vspace{10pt}

We shall present the proof by spliting it in a number of smaller intermediate results, which will be later collected in \S\ref{sec:proofOfComparison} below.

\subsection{Properties of solutions}

In this section we clarify what we mean by the parabolic and the elliptic problems. We refer the reader to \cite{CIL} for the precise definitions and the detailed overview of the standard viscosity theory.

\begin{definition}
\label{def:parabolicElliptic}
We say that $w(x,t)$ is a \emph{solution} (resp. subsolution, supersolution) of the parabolic problem in an open set $Q' \subset \Rn\times \R$ if $w$ is the standard viscosity solution (resp. subsolution, supersolution) of
\begin{align*}
w_t - F(D^2w, Dw, w) &= 0 & &\text{in } Q'.
\end{align*}

Similarly, we call $w(x)$ a \emph{solution} (resp. sub/supersolution) of the elliptic problem in an open set $\Omega' \subset \Rn$ if $w$ is the standard viscosity solution (resp. sub/supersolution) of
\begin{align*}
- F(D^2w, Dw, w) &=0 & &\text{in } \Omega'.
\end{align*}
\end{definition}

Now with the help of the previous definition, we can rigorously express the intuitive fact that the problem \eqref{RP} is parabolic in the positive phase and elliptic in the negative phase.

\begin{lemma}
\label{le:positiveNegativePart}
If $u$ is a viscosity subsolution of \eqref{RP} in $Q$ then $u_+$ is a subsolution of the parabolic problem in $Q$. Similarly, if $v$ is a viscosity supersolution of \eqref{RP} in $Q$ then $-v_-(\cdot, t)$ is a supersolution of the elliptic problem in $\Omega$ for each $t > 0$.
\end{lemma}
\begin{proof}

1. The first claim follows easily since $u_+ = \max \set{u, 0}$ and $0$ solves 
\begin{align*}
u_t - F(D^2 u, Du, u) = 0
\end{align*}
(see \eqref{eq:Fzero}). 

\vspace{10pt}

2. To prove the second claim for  $-v_- := \min \set{v, 0}$, suppose it is not a supersolution of $-F(D^2u, Du, u)=0$ at time $t=t_0>0$.
Let us denote $\zeta(x) = -v_-(x,t_0)$.
Then there is a function $\vp \in C^2$, $\vp < 0$, and an open (space) ball $B \subset \set{\zeta < 0}$ such that $-F(D^2\vp, D\vp, \vp) < 0$ on $B$, $\vp \leq \zeta$ in $B$ and $\vp < \zeta$ on $\partial B$, but $\vp = \zeta$ at some point in $B$.
Due to the continuity of $F$, there exists $\eta_0 > 0$ such that $-F(D^2 \vp, D\vp, \vp - \eta) < 0$ in $B$ for all $\eta \in [0, \eta_0]$.
For any $\de > 0$ we set
$$
Q_\de := B \times (t_0-\de, t_0].
$$
Since $v$ is lower semi-continuous, there exists small $\de>0$ such that $\cl{Q_\de} \subset \set{v < 0}$, $\vp < v(\cdot, t)$ on $\partial B$ for $t \in [t_0 - \de, t_0]$, and $v(x,t) - \vp(x) > - \eta_0$ on $\cl{Q}_\de$.
Let us define the barrier $\psi(x,t) = \vp(x) + \frac{\eta_0}{\de} t$.
Observe that $-F(D^2 \psi, D \psi, \psi) < 0$ and $\psi < 0$ in $Q_\de$ and therefore it is a strict classical subsolution of \eqref{RP}.
Furthermore, $\psi < v$ on $\partial_P Q_\de$, while $\psi = v$ at some point in $\cl Q_\de$.
This contradicts the fact that $v$ is a supersolution of \eqref{RP}.
\end{proof}

\subsection{Regularizations}
\label{sec:regularizations}

In this section, we define regularizations of the subsolution $u$ and the supersolution $v$, and prove some of their properties that are applied in the proof of the comparison theorem.

For given $r \in (0,1)$, define the regularizations
\begin{align}
\label{eq:regularizedSolutions}
\begin{aligned}
Z(x,t) &:= \sup_{\cl \Xi_r(x,t)} u , \\
W(x,t) &:= \inf_{\cl \Xi_r(x,t)} v.
\end{aligned}
\end{align}
Functions $Z$ and $W$ are well-defined on the parabolic cylinder $\cl Q_r$,
\begin{align*}
Q_r &:= \set{(x,t) \in Q \mid \cl\Xi_r(x,t) \subset Q} = \Omega_r \times (r, \infty),
\end{align*}
where 
\begin{align*}
\Omega_r := \set{x \in \Omega \mid \dist(x, \Omega^c) > r + r^{2/3}}.
\end{align*}

\begin{remark}
$\Omega_r$ has the uniform exterior ball property (with radius $r + r^{2/3}$) and therefore it is a regular domain for the elliptic problem.
\end{remark}

\begin{remark}
The main advantage of regularizing over the set $\Xi_r$ instead of the ball $B_r$ is that the appropriate level sets of $Z$ and $W$ have both space-time and space interior balls, see Proposition~\ref{pr:dualPoint}.
\end{remark}

\begin{remark}
Defining $\Xi = D + E$, instead of $D + B$ as in the previous paper \cite{KP}, has the consequence that the parabolic boundary of $\Xi$ is not $C^{1,1/2}_{x,t}$ at the top flat piece $\cl{\partial^\top \Xi}$, and thus it is not a regular set for the heat equation (see \cite{CS}). We can therefore expect that the gradient of a positive caloric function in $\Xi$ which is zero on the lateral boundary of $\Xi$ will blow up at the lateral boundary as we approach $\cl{\partial^\top \Xi}$, see Lemma~\ref{le:DsupersolutionInfty}. A similar effect of vanishing gradient is expected when the solution is positive on $\Xi^c$ but zero in $\Xi$, see Lemma~\ref{le:Dsubsolution0}. This is the necessary new ingredient required in the proof of the ``finite speed of propagation'' in Lemma~\ref{le:finiteSpeed}.
\end{remark}

\begin{proposition}
\label{pr:convolutionIsSolution}
Suppose that $v$ is a visc. supersolution on $Q = \Omega \times (0,\infty)$. Then $W$ is a visc. supersolution on $Q_r$. Similarly, if $u$ is a visc. subsolution on $Q$ then $Z$ is a visc. subsolution on $Q_r$.
\end{proposition}

The strict separation of $u$ and $v$ on the parabolic boundary of $Q$ allows us to separate $Z$ and $W$ on the parabolic boundary of $Q_r$.
\begin{proposition}
\label{pr:ZWorderedOnBoundary}
Suppose that $u \in USC$ and $v \in LSC$ in $\cl Q$ such that $u < v$ on $\partial_P Q$. Then there exists $r_0 > 0$ such that $Z < W$ on $\partial_P Q_r$ for all $0 < r \leq r_0$.
\end{proposition}

\begin{proof}
Standard from semicontinuity.
\end{proof}

Arguably the most important feature of regularizations $Z$ and $W$ is the interior ball property of their level sets, as well as of the time-slices of their level sets. We formalize this fact by introducing the notion of \emph{dual points}.

\begin{definition}
\label{de:dualPoint}
Let $r > 0$, $u \in USC(\cl \Omega)$ and let $Z$ be its sup-convolution. Let  $P \in \cl{Q}_r$. We say that $P' \in \cl{Q}$ is a \emph{sup-dual point of $P$ with respect to $u$}
\begin{align*}
\text{if } P' \in \cl \Xi(P) \text{ and } u(P') = Z(P).
\end{align*}
  Let us define $\Pi^u(P)$ to be the set of all sup-dual points of $P$ with respect to $u$,
  \begin{align*}
  \Pi^u(P) = \set{P' \in \cl \Xi_r(P) : u(P') = Z(P)}.
  \end{align*}
 
 Similarly we can define \emph{inf-dual points} $\Pi_v(P)$ for $v \in LSC(\cl \Omega)$ by $\Pi_v(P) := \Pi^{-v}(P)$.
\end{definition}

In what follows, we shall use the following convenient notation for various level sets of $u$ and $Z$:
\begin{align*}
\set{u \geq 0} &= \set{(x,t) \in \cl{Q}: u(x,t) \geq 0},\\
\set{u < 0} &= \set{(x,t) \in Q: u(x,t) < 0},
\end{align*}
which is contrasted with
\begin{align*}
\set{Z \geq 0} &= \set{(x,t) \in \cl{Q_r}: Z(x,t) \geq 0},\\
\set{Z < 0} &= \set{(x,t) \in Q_r: Z(x,t) < 0}.
\end{align*}
Sets $\set{v \leq 0}$, $\set{v > 0}$, $\set{W \leq 0}$ and $\set{W > 0}$ are defined in a similar fashion.
This choice guarantees that sets with $\geq$ and $\leq$ are closed, while sets with $<$ and $>$ are open. 

We first make a few simple observations about $\Pi^u$ and $\Pi_v$.

\begin{proposition}
\label{pr:dualPoint}
Let $u \in USC(\cl Q)$. Then for all $P \in \cl Q_r$:
\begin{enumerate}
\romanlist
\item $\Pi^u (P) \neq \emptyset$,
\item $\cl \Xi_r(P') \cap \cl Q_r \subset \set{Z \geq Z(P)}$ for all $P' \in \Pi^u(P)$,
\item if $P \in \partial \set{Z \geq 0}$, then $\Pi^u(P) \subset \partial \Xi_r(P) \cap \partial \set{u \geq 0}$ and $\Xi_r(P) \subset \set{u < 0}$. 
\end{enumerate}
\end{proposition}

Since the closed sets $\set{W \leq 0}$ and $\set{Z \geq 0}$ have closed space and space-time interior balls at each point, they don't have any points that are isolated from their interior:
\begin{lemma}
\label{le:nohair}
The level sets of the functions $W$ and $Z$ defined above have the following properties:
\begin{align*}
\cl{\interior \set{W \leq 0}} &= \set{W \leq 0}, &
\cl{\interior \set{Z \geq 0}} &= \set{Z \geq 0}.
\end{align*}
Moreover it is true for every time-slice $t$ (with space closure and interior).
\end{lemma}
\begin{proof}
This follows from the interior ball property in Proposition~\ref{pr:dualPoint}. Fix time $t$ and for simplicity define $E = \set{W \leq 0}_t$. Pick any point $x \in E$. Then there is an open ball $B = B_r(y)$ such that $x \in \cl B$ and $\cl B \subset E$. But $B$ is open so $B \subset \interior E$ and therefore $x \in \cl B \subset \cl{\interior E}$.
\end{proof}

Another important property of the regularized supersolution $W$ is that each point of the time-slice $\set{W \leq 0}_t$ is connected to the boundary $\partial \Omega_r$ with nonpositive values of $W$ by a wide ``trunk'' of finite length:
\begin{lemma}
\label{le:trunk}
Let $W$ be the inf-convolution of a supersolution $v$. Then for every $(\xi,\si) \in \set{W \leq 0}$ there is a piecewise linear continuous curve $\ga: [0,1] \to \cl \Omega_r$, with finite length such that $\ga(1) \in \Omega_r^c$ and 
\begin{align*}
\xi \in \bigcup_{s \in [0,1]} \cl B_{r/2}(\ga(s)) \cap \cl \Omega_r \subset \at{\set{W \leq 0}}{t=\si}.
\end{align*}
\end{lemma}

\begin{proof}
Pick $P = (\xi, \si) \in \set{W \leq 0}$. By Proposition~\ref{pr:dualPoint}, there exists a point $P' = (\xi', \si')$ on $\Pi_v(P)$ such that $(\xi, \si) \in \cl \Xi_r(\xi',\si') \subset \set{W \leq 0}$. Therefore we can also find $\hat \xi$ such that $\xi \in \partial B_{r/2}(\hat \xi)$ and $\cl B_{r/2} (\hat \xi) \subset \set{W \leq 0}_\si$. The situation is depicted in Figure~\ref{fig:trunk}.
\begin{figure}
\centering
\fig{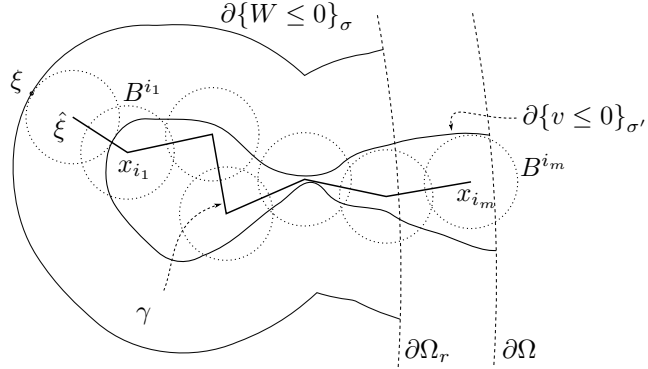}{4.5in}
\caption{Space cross section of the situation in the proof of Lemma~\ref{le:trunk}}
\label{fig:trunk}
\end{figure}

Let us denote $H = \at{\set{v\leq 0}}{t=\si'}$. Since $H$ is compact we can select a finite subcover from the open cover
\begin{align*}
H \subset \bigcup_{x \in H} B_{r/2}(x).
\end{align*}
Let us denote the balls in the finite subcover by $B^1, \ldots, B^k$, with centers $x_1, \ldots, x_k$.

Suppose that there exists a permutation $\pth{j_1, \ldots, j_k}$ of $\pth{1, \ldots, k}$ and $q \in \N$ such that
\begin{align*}
\pth{\bigcup_{1\leq l \leq q} B^{j_l}} \cap \pth{\partial \Omega \cup \bigcup_{q < l\leq k} B^{j_l}} = \emptyset.
\end{align*}
Let us denote $G = \bigcup_{1\leq l \leq q} B^{j_l}$ and $C = H \cap G$. Note that $G$ is open and $C$ is nonempty.
By definition, $v > 0$ on $\partial G \times \set{\si'}$ and therefore there is $\de > 0$ such that $v > 0$ on $\partial G \times (\si' - \de, \si']$ ($\set{v > 0}$ is open by $v \in LSC$).
In particular, since $-v_-(\cdot, t)$ is a supersolution of the elliptic problem for every $t$ (Lemma~\ref{le:positiveNegativePart}), an application of the elliptic comparison principle (Proposition~\ref{pr:ellipticComparison}) yields that $v \geq 0$ in $G \times (\si' - \de, \si']$.
And a straightforward barrier argument (barrier constructed in Lemma~\ref{le:arbitrarySpeedBarrier}) shows that in fact $v > 0$ on $G \times \set{\si'}$. This is a contradiction with $v(x_{j_1}, \si') \leq 0$ at $x_{j_1} \in G$.

Therefore we conclude that we can find $i_1, \ldots, i_m$ distict with $\xi' \in B^{i_1}$, $B^{i_m} \cap \partial \Omega \neq \emptyset$ and
\begin{align*}
B^{i_l} \cap B^{i_{l+1}} \neq \emptyset, \qquad 1 \leq l < m.
\end{align*}
Since we observe that $\cl B_r(x_i) \subset \set{W \leq 0}_{\si}$ for all $1 \leq i \leq k$, we can choose the curve $\ga$ as the 
piecewise linear curve connecting points $\hat \xi, x_{i_1}, \ldots, x_{i_m}$.
\end{proof}

Now we present two results, Lemmas~\ref{le:DsupersolutionInfty} and \ref{le:Dsubsolution0}, that justify the shape of the domain $\Xi_r$ chosen in the definition of the regularizations $Z$ and $W$.

\begin{lemma}
\label{le:DsupersolutionInfty}
Let $v$ be a visc. supersolution of \eqref{RP} on $\cl\Xi$, $\Xi = \Xi_r(\xi, \si)$ for some $(\xi, \si) \in \Rn \times \R$, and $v > 0$ in $\Xi$. Then there exists $f \in C([0,r])$, $f(0) = 0$, $f > 0$ on $(0,r]$ and $\frac{f(s)}{s} \to \infty$ as $s \to 0+$ such that
\begin{align*}
v(x, t) \geq f(r - \abs{x - \xi}) \quad \text{for } \abs{x-\xi} < r,\ t \in [\si + r/2, \si + r].
\end{align*}
\end{lemma}

\begin{proof}
We only prove that $v(x,\si + r) \geq f(r - \abs{x - \xi})$.
The full result follows from the simple observation
\begin{align*}
\Xi_{r/2}(x,t) \subset \Xi_r \qquad \text{for } \abs{x} \leq \frac{r}{2}, \ t \in \bra{-\frac{r}{2}, \frac{r}{2}}.
\end{align*}

Let us fix a point $(\zeta, \si + r) \in D_r(\xi, \si +r)$.
Since the argument is invariant under translation and space rotation, we can assume that $(\zeta, \si+r) = (0,0)$. Let $s = r - \abs{\zeta - \xi}$.
Our goal is to show that $v(0,0) = v(\zeta, \si + r) \geq f(s)$ for some function $f$.
The situation is depicted in Figure~\ref{fig:expansionWleq0}.
\begin{figure}
\centering
\fig{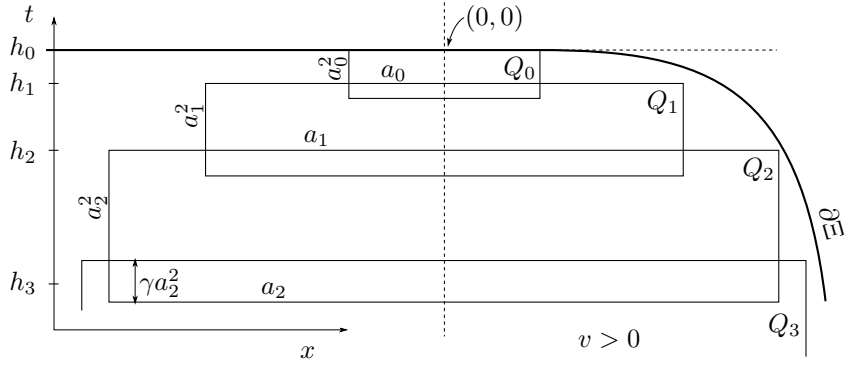}{4.5in}
\caption{Iterations in the proof of lower bound in Lemma~\ref{le:DsupersolutionInfty}}
\label{fig:expansionWleq0}
\end{figure}

It is straightforward to estimate the distance of the lateral boundary  $\partial_L \Xi$ at each time, $\at{\partial \Xi}{t}$, from the origin $x = 0$ for $t \in (-r,0)$:
\begin{align}
\label{eq:cubicXi}
\dist(0, \at{\partial \Xi}{t}) = s + \sqrt[3]{r^2 - (t + r)^2} = s +\sqrt[3]{-2rt - t^2} > s + \sqrt[3]{-rt}.
\end{align}

Since $v \in LSC$ is positive in $\Xi$, it has a positive minimum on the compact set 
\begin{align*}
K := \set{(x,t) : \abs{x - \xi} \leq r + \sqrt[3]{-r t},\ t \in [-7r/8, -r^2/256]}
\end{align*}
as $K \subset \Xi$.

Let us set 
\begin{align*}
\ga := \min \set{\frac{1}{16n\la + 8 \de_1 + 2\de_0}, 1}
\end{align*}
and let us define the barrier
\begin{align}
\label{decrParabolaBarrier}
\vp(x,t) = \pth{-\frac{1}{2\ga} t - 4 \abs{x}^2 +1}_+.
\end{align}
Clearly $\vp$ is a viscosity subsolution of the parabolic problem due to \eqref{eq:structural}, and $\vp(0, \ga) = \frac{1}{2} \vp(0,0) = \frac{1}{2}$.
Let us define a sequence of parabolic cylinders
\begin{align*}
 Q_j = B_{a_j} \times (h_j - (1 + \ga) a_j^2, h_j), \quad j \geq 0,
 \end{align*} 
 where 
 \begin{align*}
 h_j &= -\sum_{i=0}^{j-1} a_i^2, \quad j \geq 1, \qquad h_0 = 0.
 \end{align*}
$a_j$'s are defined in accordance with \eqref{eq:cubicXi} as
\begin{align*}
  a_j = \sqrt[3]{-r h_j} + s, \quad j \geq 1, \qquad a_0 = \min \set{s, \frac{r}{16}}.
 \end{align*}
Writing $h_{j+1} = h_j - a_j^2$, we can derive the recurrence relation 
\begin{align}
 \label{recurak}
 a_{j+1} = \sqrt[3]{r a_j^2 + (a_j - s)^3} + s, \quad j \geq 0.
\end{align}
 Note that $a_j$'s were chosen in such a way that $Q_j \subset \Xi$ for all $j$ for which $h_{j+1} < r$ due to \eqref{eq:cubicXi}.
 
First assume that $s < r/16$. From \eqref{recurak} we estimate
 \begin{align*}
 a_{j+1} \geq r^{1/3} a_j^{2/3}, \quad j \geq 0,
 \end{align*}
 which yields
 \begin{align}
 \label{eq:boundak}
 a_j \geq r \pth{\frac{s}{r}}^{\pth{\frac{2}{3}}^j} \to  r \quad \text{ as } j \to \infty.
 \end{align}
 
We want to estimate the minimal $k = k(s)$ that guarantees
\begin{align*}
\partial_B Q_k = D_{a_k} (0, h_k - (1 + \ga) a_k) \subset K.
\end{align*}
We first observe that if $a_{j} \geq \frac{r}{2} + s$ for some $j \geq 0$ then $k \leq j$. Therefore \eqref{eq:boundak} yields the upper bound
\begin{align}
\label{kBound}
k \leq \frac{\log\left[\frac{\log \frac{s}{r}}{\log \frac{1}{2}}\right]}{\log\frac{3}{2}} + 1.
\end{align}

To show that such $k$ indeed exists, suppose that $h_j > - r^2/8$ for some $j$. Then $a_j < r/2 + s$ and 
\begin{align*}
h_j - (1 + \ga)a_j^2 \geq -\frac{r^2}{8} - 2 \pth{\frac{r}{2} +s}^2 = - \frac{r^2}{8} - \frac{r^2}{2} - 2rs - 2s^2 \geq -\frac{7r}{8},
\end{align*}
since $s < r/16$ and $\ga \leq 1$.

If $s \geq r/16$ then we set $k  = 0$. 
 
 \vspace{10pt}
 
Next we iteratively define, for $j = k, \ldots, 0$, the rescaled barriers
\begin{align*}
\vp_j(x,t) = \kappa_j \vp(a_j^{-1}x, a_j^{-2}(t-h_j) + 1 + \ga),
\end{align*}
where 
\begin{align*}
\kappa_j =
\begin{cases} \min_K v > 0 & j = k,\\
\min_{B_{a_j/2}} \psi_{j+1}(\cdot, h_{j+1} - \ga a_j^2)  & j < k,
\end{cases}
\end{align*}
and the functions $\psi_j$, $j = k, \ldots, 0$, are the unique solutions of the parabolic problem with boundary data $\psi_j = \vp_j$ on $\partial_P Q_j$.
We observe that $\psi_j \geq \vp_j$ in $Q_j$ due to the parabolic comparison (Prop.~\ref{pr:parabolicComparison}) since $\vp_j$ is a subsolution of the parabolic problem.

Now the parabolic Harnack inequality (Prop.~\ref{pr:Harnack} applied with $t_1 = \ga$ and $t_2 = 1 + \ga$) yields
\begin{align}
\label{kappaIteration}
\begin{aligned}
\kappa_{j-1} &= \inf_{B_{a_{j-1}/2}} \psi_j(\cdot, h_j - \ga a_{j-1}^2)
\geq \inf_{B_{a_j/2} \times [h_j - \ga a_{j-1}^2, h_j]} \psi_j
\\&\geq \ov{c} \sup_{B_{a_j/2}} \psi_j(\cdot, h_{j+1})
\geq \frac{1}{2c} \sup_{B_{a_j/2}} \psi_j(\cdot, h_{j+1} - \ga a_j^2)
\\&= \frac{1}{2c} \kappa_j.
\end{aligned}
\end{align}

Finally, we realize that, by definition, $v \geq \vp_k = \psi_k$ on $\partial_P Q_k$ and thus an application of the parabolic comparison shows that $v \geq \psi_k$ in $Q_k$.
Since we chose $\psi_j$ so that $\psi_j \leq \psi_{j+1}$ on $Q_j \cap Q_{j+1}$, we can apply the comparison principle iteratively to conclude that $v \geq \psi_j$ in $Q_j$ for $j = k, \ldots, 1$.
The case $j = 0$, however, has to be considered separately, because we only know that $v > 0$ in $\Xi$, i.e. in $Q_0 \cap \set{t < 0}$.
Therefore the parabolic comparison can only show that $v \geq \psi_0$ in $Q_0 \cap \set{t < 0}$.
Nevertheless, since $\psi_0(0,0) > 0$ and $v$ is bounded from below on $\cl\Xi$, a straightforward barrier argument using a strict classical subsolution of \eqref{RP}, that can be constructed in the form similar to \eqref{decrParabolaBarrier}, extended in negative phase using a strict subsolution of the elliptic problem, shows that $v(0,0) \geq \psi_0(0,0)$.

Now let $\al = \ov{2c}$. Then a simple induction of \eqref{kappaIteration} using the bound \eqref{kBound} yields the lower bound
 \begin{align*}
v(0, 0) \geq \psi_0(0,0) \geq \al \pth{\frac{\log \frac{s}{r}}{\log \frac{1}{2}}}^{\frac{\log \al}{\log \frac{3}{2}}} \min_K v =: f(s), \quad s \in (0, r].
 \end{align*}
A straightforward computation verifies that $f(s) \to 0$ and $f(s)/ s \to \infty$ as $s \searrow 0+$ since $\al < 1$.
\end{proof}

\begin{corollary}
[Continuous expansion of $\set{v \leq 0}$]
\label{co:continuityWleq0}
If $v$ is a visc. supersolution of \eqref{RP} then the set $\set{v \leq 0}$ cannot expand discontinuously, i.e.
\begin{align*}
\cl{\set{v \leq 0, \ t <  s}} = \set{v\leq 0, \ t\leq s},
\end{align*}
for any $s$.
\end{corollary}
\begin{proof}
Suppose that the claim is not true, i.e. there is a point 
\begin{align*}
(\xi,\tau) \in \set{v\leq 0, \ t\leq \tau} \setminus \cl{\set{v \leq 0, \ t <  \tau}}.
\end{align*}
That means that there is $\rho > 0$ such that 
\begin{align*}
 \cl\Xi_\rho(\xi, \tau - \rho) \cap \set{v \leq 0, \ t <  \tau} = \emptyset
\end{align*}
and thus $v > 0$ in $\Xi_\rho(\xi, \tau - \rho)$.  Lemma~\ref{le:DsupersolutionInfty} then yields $v(\xi, \tau) > 0$, a contradiction.
\end{proof}

\begin{remark}
Note that the set $\set{u \geq 0}$ of a subsolution $u$ can expand discontinuously when a part of the set $\set{u < 0}$ is pinched off by a collision of two ``fingers'' of $\set{u \geq 0}$.
\end{remark}

\begin{lemma}
\label{le:Dsubsolution0}
Let $u \geq 0$ be a bounded subsolution of the parabolic problem in a parabolic neighborhood of $\cl \Xi$, with $\Xi = \Xi_r(\xi, \si)$ for some $(\xi, \si) \in \Rn \times \R$ and $r \in (0,1)$. Assume that $u = 0$ on $\Xi$. Then there exists $\e > 0$ and $g \in C^1([0, r + \e))$ with $g = g' = 0$ on $[0,r]$ such that
\begin{align*}
0 \leq u(x,t) \leq g(\abs{x-\xi}) \qquad \text{for } \abs{x - \xi} < r +\e,\  t \in [\si, \si + r].
\end{align*}
\end{lemma}

\begin{proof}
Let us choose $(y,s) \in B_r(x) \times [\si, \si + r]$.
For given $\e > 0$, $\eta \in (0, \e)$, we define the barrier
\begin{align*}
\psi_{\e,\eta}(x,t) &:= \frac{4M}{\e}(4n\Lambda t + \abs{x}^2 + \eta),
\end{align*}
on
\begin{align*}
 E_{\e, \eta} &:= \pth{B_{\e^{1/2}} \times (-\frac{\e}{8n\Lambda}, 0]} \cap \set{\psi_{\e,\eta}(x,t) > 0}.
\end{align*}
Note that:
\begin{enumerate}
\romanlist
\item Due to \eqref{eq:structural},
$\psi_{\e,\eta}$ is a strict supersolution of the parabolic problem in $E_{\e,\eta}$ as long as $\e^{1/2} < \frac{2n\Lambda}{3 (\de_0 + \de_1)}$, and $\psi_{\e,\eta} \geq 2M$ on $\partial B_{\e^{1/2}} \times [- \frac{\e}{8n\Lambda},0]$.
\item $\psi_{\e,\eta} > 0$ if and only if $t > - \frac{\eta}{4n\Lambda}$ or $\abs{x} > \sqrt{-4n\Lambda t - \eta}$.
\item $\sqrt[3]{-rt} > \sqrt{-4n\Lambda t}$ when $t \in (-\frac{r}{8n\Lambda},0)$ and $\sqrt[3]{-\frac{r\e}{8n\Lambda}} > \e^{1/2}$ for small $\e > 0$.
\end{enumerate}
(i)--(iii) verify that for small $\e > 0$ and all $\eta \in (0,\e)$,
we have the ordering 
$$u(x,t) \leq \psi_{\e,\eta} (x-y, t-s)\hbox{ on }\partial_P E_{\e,\eta} + (y,s)$$
and hence by the parabolic comparison (Prop.~\ref{pr:parabolicComparison}) the ordering holds in $E_{\e, \eta} + (y,s)$.
Therefore there exists a constant $\e> 0$ such that
\begin{align*}
u(x,t) \leq g(\abs{x - \xi}) := \inf_{\substack{\eta>0\\\zeta\in B_r}} \psi_{\e, \eta}(x- \zeta, 0)
\end{align*}
for $(x,t) \in B_{r + \e} \times [\si, \si + r]$.
\end{proof}

\subsection{\texorpdfstring{$Z$}{Z} and \texorpdfstring{$W$}{Z} cross}

We proceed with the the proof of the comparison principle (Theorem~\ref{th:comparisonPrinciple}) for the regularized solutions $Z$ and $W$, in place of $u$ and $v$. To argue by contradiction, we investigate the situation when $Z$ and $W$ cross in $Q_r$, i.e. there is a finite first crossing time $t_0$, defined by
\begin{align}
\label{eq:crossingTime}
t_0 := \sup \set{\tau \mid Z(\cdot, t) < W(\cdot, t) \text{ for } 0 \leq t \leq \tau}.
\end{align}

Now since $\set{W \leq 0}$ cannot expand discontinuously due to Lemma~\ref{co:continuityWleq0}, we can prove that a certain ordering of level sets of $W$ and $Z$ is preserved up to the crossing time $t_0$.

\begin{lemma}
\label{le:negativeSetsOrdered}
Let $Z$ and $W$ be the regularized solutions defined in \eqref{eq:regularizedSolutions} and $t_0$ be the crossing time defined in \eqref{eq:crossingTime}. Then
\begin{align}
\label{eq:touchOnBoundary}
\set{Z \geq 0}_{t_0} \cap \set{W \leq 0}_{t_0} = \partial \pth{\set{Z \geq 0}_{t_0}} \cap \partial \pth{\set{W \leq 0}_{t_0}}
\end{align}
In particular,
\begin{align}
\label{eq:disjointInterior}
\interior \pth{\set{Z \geq 0 }_{t_0}} \cap \interior \pth{\set{W \leq 0}_{t_0}} = \emptyset.
\end{align}
\end{lemma}

\begin{proof}
First observe that $\set{W \leq 0}_t \subset \set{Z < 0}_t$ for $t < t_0$ (or equivalently $\set{Z \geq 0}_t \subset \set{W > 0}_t$ for $t < t_0$).

\medskip\noindent
\textbf{Step 1.}
We claim that $\interior \set{W \leq 0}_{t_0} \subset \set{Z < 0}_{t_0}$.

Pick any $\xi \in \interior \set{W \leq 0}_{t_0}$. Due to continuous expansion of $\set{W \leq 0}$ (Corollary~\ref{co:continuityWleq0}), there exists $\de > 0$ such that $B_\de(\xi) \subset \interior \set{W \leq 0}_t$ for a short time before $t_0$, $t \in [t_0 - \de, t_0]$.
Lemma~\ref{le:trunk} yields that $B_\de(\xi)$ is connected to $\partial \Omega_r$ with non-positive values of $W$ (and thus negative $Z$) through a wide ``trunk'' 
\begin{align*}
G = (\ga([0,1]) + B_{r/2}) \cap \Omega_r,
\end{align*}
and $B_\de(\xi) \subset G \subset \set{W \leq 0}_{t_0}$.  
The continuous expansion of $\set{W \leq 0}$ again guarantees that $G$ can be chosen so that $G \subset \set{W \leq 0}_t$ for a short time before $t_0$.
Finally, we recall that $Z < W$ in $G \times [0, t_0)$.
Therefore we can construct a strict classical supersolution of \eqref{RP} up to the time $t_0$ that will stay above $Z$ and is negative in $B_\de(\xi)$. Indeed, we solve the elliptic problem in $G$ with zero on $\partial G \cap \Omega_r$ and negative data on $\partial G \cap \partial \Omega_r$. Since $G$ has a uniform interior ball property at points $\partial G \cap \Omega_r$, we can proceed as in the proof of Lemma~\ref{stability:initial}.
We conclude $Z(\xi, t_0) < 0$.

\medskip\noindent
\textbf{Step 2.} 
Since $\set{Z \geq 0} \cap Q_r = \set{Z < 0}^c \cap Q_r$, we have
\begin{align}
\label{int-Wleq0-cap-Zgeq0}
\interior \pth{\set{W \leq 0}_{t_0}} \cap \set{Z \geq 0}_{t_0} = \emptyset,
\end{align}
which together with Lemma~\ref{le:nohair} also gives
\begin{align}
\label{int-Zgeq0-cap-Wleq0}
\interior \pth{\set{Z \geq 0}_{t_0}} \cap \set{W \leq 0}_{t_0} = \emptyset.
\end{align}

\medskip\noindent
\textbf{Step 3.}
Now \eqref{int-Wleq0-cap-Zgeq0} clearly implies
\begin{align*}
\set{Z \geq 0}_{t_0} \cap \set{W \leq 0}_{t_0} \subset \set{W \leq 0}_{t_0} \setminus \interior \pth{\set{W \leq 0}_{t_0}} = \partial \pth{\set{W \leq 0}_{t_0}}.
\end{align*}
Using \eqref{int-Zgeq0-cap-Wleq0} symmetrically concludes the proof of the lemma.
\end{proof}

\begin{corollary}
\label{co:unitNormal}
Let $Z, W$ and $t_0$ be defined as above and let 
\begin{align*}
\xi \in \set{Z \geq 0}_{t_0} \cap \set{W \leq 0}_{t_0}.
\end{align*}
Then there is a unique unit vector $\nu$ such that $\nu$ is the unit outer normal to $\set{Z \geq 0}_{t_0}$ at $\xi$ and the unit outer normal vector to $\cl {\set{W > 0}}_{t_0}$ at $\xi$.
\end{corollary}

\begin{proof}
Due to Lemma~\ref{le:negativeSetsOrdered}, $\xi \in \partial \pth{\set{Z \geq 0}_{t_0}} \cap \partial \pth{\set{W \leq 0}_{t_0}}$. 
Moreover there are dual points 
$$(\xi_u, \si_u) \in \Pi^u(P) \subset  \partial\set{u \geq 0}\hbox{ and }(\xi_v, \si_v) \in \Pi_v(P) \subset\partial \set{v \leq 0}$$ (see Proposition~\ref{pr:dualPoint}). 

\vspace{10pt}

Let $B_u = \interior \pth{\at{\cl \Xi_r(\xi_u, \si_u)}{t = t_0}}$ and $B_v = \interior \pth{\at{\cl \Xi_r(\xi_v, \si_v)}{t = t_0}}$. Due to the definition of $\Xi_r$, $B_u$ and $B_v$ are balls in $\Rn$, centered at $\xi_u$ and $\xi_v$, respectively, of radius greater than or equal to $r$. Observe that $B_u \subset \interior \set{W \geq 0}_{t_0}$ and $B_v \subset \interior \set{Z \leq 0}_{t_0}$. Therefore
\begin{align*}
B_u \cap B_v \subset \pth{\interior \set{W \geq 0}_{t_0}} \cap \pth{\interior \set{Z \leq 0}_{t_0}} = \emptyset,
\end{align*}
while $\cl B_u \cap \cl B_v = \set{\xi}$.
This is true for an arbitrary choice of dual points. Therefore the outer unit normal of $\nu$ of $B_u$ (resp. inner unit normal of $B_v$) at $\xi$ is uniquely determined, and can be taken as
\[
\nu = \frac{\xi - \xi_u}{\abs{\xi - \xi_u}} = \frac{\xi_v - \xi}{\abs{\xi_v - \xi}}. \qedhere
\]
\end{proof}

\subsection{Finite speed of expansion}

Our goal in this section is to use the ordering of the support to prove the ordering of the functions $Z$ and $W$ at the contact time $t = t_0$. This needs a careful analysis since $Z$ and $W$ are merely semi-continuous. 

\vspace{10pt}

From Lemma~\ref{le:positiveNegativePart} we know that $Z_+$ is a parabolic subsolution, and $-W_-$ is an elliptic supersolution for each time, and also $W$ is a parabolic supersolution in $\set{W > 0}$ (open set) and $Z$ is an elliptic subsolution for each time in $\set{Z < 0}$ (open set). Therefore we can invoke the standard comparison principle (Propositions~\ref{pr:ellipticComparison} and \ref{pr:parabolicComparison}) in the open sets $\set{W > 0}$ and $\set{Z < 0}_t$, if we know that the functions are ordered on the boundaries of these sets. This is not completely obvious, however, and we have to pay special attention to the situation at the contact point.

\begin{lemma}[Finite speed of expansion at the contact point]
\label{le:finiteSpeed}
Let $Z$, $W$ be the regularizations and $t_0$ be the crossing time as defined in \eqref{eq:crossingTime}. Then for any $\xi \in \set{Z \geq 0}_{t_0} \cap \set{W \leq 0}_{t_0}$ 
\begin{align*}
\Pi^u(\xi,t_0) \cup \Pi_v(\xi,t_0) \subset \partial_L \Xi_r(\xi, t_0).
\end{align*}
\end{lemma}

\begin{proof}
We split the proof into a number of shorter steps.  Here we use the simplified notation $\Xi = \Xi_r(\xi, t_0)$ (and analogously for $\Pi^u$ and $\Pi_v$).

\medskip\noindent
\textbf{Step 1.} $\Pi_v \cap \cl{\partial^\top \Xi} = \emptyset$.

For the sake of the argument, suppose that this does not hold and there indeed is a point $P' = (\xi', \si') \in \Pi_v \cap \cl{\partial^\top \Xi}$. Note that $\si' = t_0 +r$. We observe that $v > 0$ in $\Xi$ due to \eqref{eq:touchOnBoundary} and Proposition~\ref{pr:dualPoint}(iii), and thus we can apply Lemma~\ref{le:DsupersolutionInfty} for $v$ on $\cl \Xi$. The first consequence is $\abs{\xi' - \xi} = r$, i.e. 
$$
P' \in \partial D_r (\xi, t_0 + r).
$$

Let $\nu$ be the unit normal from Corollary~\ref{co:unitNormal}. We choose $\rho_0 = \min(\frac{r}{4}, \hat \rho_c)$ ($\hat \rho_c$ is defined in Prop.~\ref{pr:radialSolution}) and set
$$
\zeta = \xi' - \rho_0 \nu,\ \hat \si = \si',\ \hat a = 2,\ \hat b = -1\hbox{ and }\hat \om = 0.
$$
 Proposition~\ref{pr:radialSolution} provides us with a radially symmetric subsolution of \eqref{RP} with parameters $\hat a$, $\hat b$ and $\hat \omega$. We denote $\hat \vp$ its translation by $(\zeta, \hat \si)$ which is defined on the cylinder 
\begin{align}
\label{cylinder-K}
K = \set{\rho_0 - \e < |x - \zeta| < \rho_0 + \e,\ t \in [\hat \si - \e, \hat \si]},
\end{align}
for some $\e > 0$. The situation is depicted in Figure~\ref{fig:finspeedv}, which is a part of the larger picture in Figure~\ref{fig:contact}.
\begin{figure}
\centering
\fig{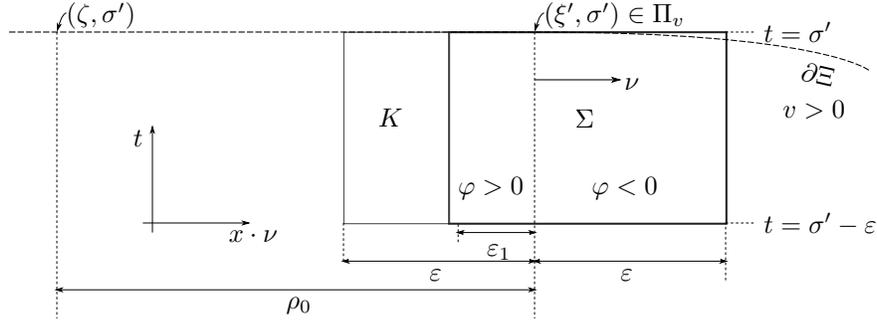}{4.5in}
\caption{Construction of a test function $\vp$ in Lemma~\ref{le:finiteSpeed}}
\label{fig:finspeedv}
\end{figure}

Let us define $\vp = \mu \hat \vp$, where $\mu > 0$ is picked large enough to ensure that $\min_{\cl {K}} \vp < \min_{\cl {K}} v$.
Furthermore, since $v(x,t) \geq f(r - \abs{x - \xi})$ for $\abs{x - \xi} < r$, with $f(s) /s \to \infty$ as $s \to 0+$, there exists $\e_1 > 0$ for which $v > \vp$ on
\begin{align*}
\set{\rho_0 - \e_1 \leq |x - \zeta| < \rho_0,\ t \in [\hat \si - \e, \hat \si]}.
\end{align*}
We shall compare $v$ and $\vp$ on a cylinder $\Sigma$,
\begin{align*}
\Sigma := \set{\rho_0 - \e_1 < |x - \zeta| < \rho_0 + \e, \ t \in (\hat \si - \e, \hat \si]} \subset K.
\end{align*}
So far we have shown that $v > \vp$ on all of $\partial_P \Sigma$ except on 
\begin{align*}
\set{(x, \hat \si - \e) : \rho_0 \leq \abs{x - \zeta} < \rho_0 + \e} \subset \Xi,
\end{align*}
where clearly $\vp \leq 0$ while $v > 0$. 

Therefore $\vp < v$ on $\partial_P \Sigma$ while $\vp \geq v$ at $P' \in \Sigma$, and we arrive at a contradiction since $\vp$ is a strict classical subsolution of \eqref{RP} on $\Sigma$.

\medskip\noindent
\textbf{Step 2.} $\Pi^u \cap \cl{\partial^\top \Xi} = \emptyset$

The proof of step 2 is similar to that of step 1. Choose $P' \in \Pi^u \cap \cl{\partial^\top \Xi}$. Due to \eqref{eq:disjointInterior}, Lemma~\ref{le:trunk} and Hopf's lemma (Prop.~\ref{pr:Hopf}), we have 
\begin{align*}
\liminf_{h \searrow 0+} \frac{-Z(\xi + \nu h)}{h} = 2c > 0.
\end{align*}

Therefore there is $h_0 > 0$ for which
\begin{align}
\label{eq:upperBoundByHopf}
u(x,t) \leq Z(\xi + \nu h, t_0) < -ch < 0 \quad \text{for } h \leq  h_0,\ (x,t) \in \cl \Xi_r(\xi + \nu h, t_0).
\end{align}
Proceeding as in the proof of step 1, we have $\hat \vp$, $\e \in (0, h_0)$ and $K$, this time with $\zeta = \xi' + \rho_0 \nu$, $\hat \sigma = \sigma'$. In contrast with step 1, we define $\vp = -\mu \hat \vp$, and \eqref{eq:upperBoundByHopf} lets us choose $\mu > 0$ small for which $\vp > u$ on 
\begin{align*}
\set{\abs{x- \zeta} = \rho_0 - \e , \ t \in [\hat \si - \e, \hat \si]} \cap \set{\rho_0 - \e \leq \abs{x- \zeta} \leq \rho_0, \ t = \hat \si - \e}. 
\end{align*}

Now we find $\e_1 > 0$ such that $\vp > u$ on $\set{\abs{x - \zeta} = \rho_0 + \e_1, \ t \in [\hat \si - \e, \hat \si]}$, this is again possible thanks to Lemma~\ref{le:Dsubsolution0} applied to $u$ on $\Xi$. 

Finally, define 
\begin{align*}
\Sigma := \set{\rho_0 - \e < \abs{x-\zeta} < \rho_0 + \e_1, \ t \in (\hat \si - \e, \hat \si]}.
\end{align*}
Since clearly $\vp \geq 0$ and $u < 0$ on $\set{\rho_0 \leq \abs{x- \zeta} < \rho_0 + \e, \ t = \hat \si - \e}$, we succeeded in showing that $\vp > u$ on $\partial_P \Sigma$ while $0 = \vp \leq u$ at $P' \in \Sigma$, a contradiction.

\medskip\noindent
\textbf{Step 3.} $\Pi^u \cap \cl{\partial_\bot \Xi} \neq \emptyset$ then $\Pi_v \cap \cl{\partial^\top \Xi} = \emptyset$ (and also when if we swap $\Pi^u$ and $\Pi_v$)

Indeed, this follows from \eqref{eq:disjointInterior} and the definition of $t_0$ in \eqref{eq:crossingTime}, which together yield that for any $P_u \in \Pi^u$ and $P_v \in \Pi_v$ we have 
\begin{align*}
\cl \Xi_r(P_u) \cap \cl \Xi_r(P_v) \cap \set{t \leq t_0} = (\xi, t_0) = P.
\end{align*}

\medskip\noindent
\textbf{Step 4.} To finish the proof, we simply realize that steps 2 and 3 used together imply that $\Pi^u \cap \cl{\partial_\bot \Xi} = \emptyset$, and similarly steps 1 and 3 imply $\Pi_v \cap\cl{\partial_\bot \Xi} = \emptyset$.
\end{proof}
\begin{remark}
Lemma~\ref{le:finiteSpeed} shows that the situation at any $P = (\xi, t_0) \in \set{Z \geq 0} \cap \set{W \leq 0}$ looks like Figure~\ref{fig:contact}.
\end{remark}

\begin{figure}
\centering
\fig{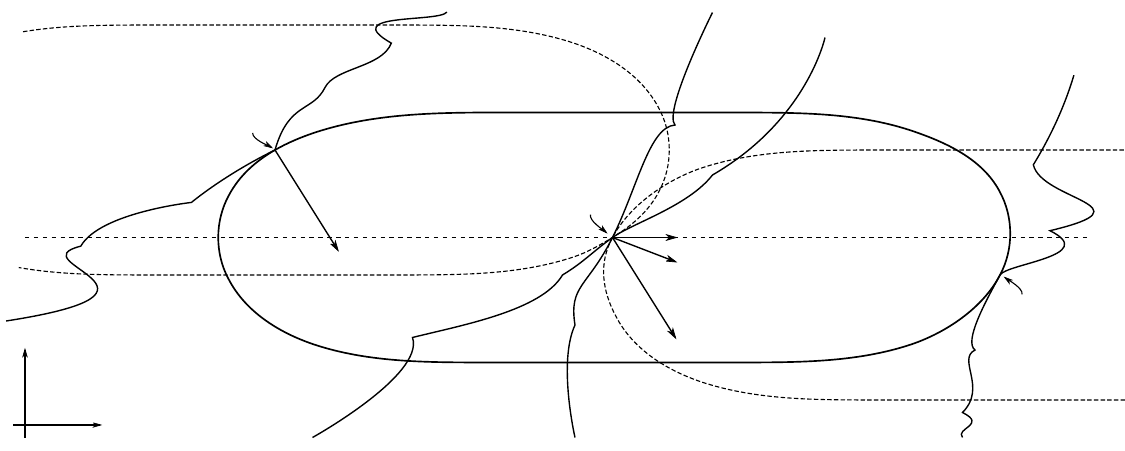}{5.4in}
\caption{Situation at $P = (\xi, t_0) \in \set{Z \geq 0} \cap \set{W \leq 0}$.}
\label{fig:contact}
\end{figure}

\vspace{15pt}

Now we have enough regularity to show the following:
\begin{lemma}
\label{le:ZWzeroAtContact}
For $Z$, $W$, $t_0$ defined above, we have 
\begin{align*}
Z = W = 0 \quad \text{on} \quad \set{Z \geq 0}_{t_0} \cap \set{W \leq 0}_{t_0}.
\end{align*}
\end{lemma}
\begin{proof}
We will only show this result for $Z$.
A similar, simpler argument applies to $W$ as well. See also \cite{KP}*{Lemma~3.6}.

Let $M = 2 \max_{\set{t \leq t_0}} Z < \infty$ and $P = (\xi,t_0) \in \set{Z \geq 0} \cap \set{W \leq 0}$.
We can also choose $P_v \in \Pi_v(P)$.
Let $\nu$ be the unit normal from Corollary~\ref{co:unitNormal}.
Lemma~\ref{le:finiteSpeed} guarantees that there exists $m_v \in \R$ such that $(\nu, m_v)$ is an interior normal of $\partial\Xi(P_v)$ at $P$, see Figure~\ref{fig:contact}.
This is why we call Lemma~\ref{le:finiteSpeed} ``finite speed of expansion''.

We recall that $Z < W \leq 0$ in $\Xi(P_v) \cap \set{t < t_0}$.
Furthermore, the argument from step 2 of the proof of Lemma~\ref{le:finiteSpeed}, using \eqref{eq:disjointInterior}, Lemma~\ref{le:trunk} and Proposition~\ref{pr:Hopf}, verifies that $Z < 0$ in $\Xi(P_v) \cap \set{t = t_0}$ as well.

As in the proof of Lemma~\ref{le:finiteSpeed}, we can construct a strict classical supersolution of \eqref{RP} with the help of Proposition~\ref{pr:radialSolution}.
Let us again choose $\rho_0$ small enough so that for $\hat a = 2$, $\hat b = -1$, $\hat \om = \max \set{0, 2m_v}$, $\zeta = \xi + \rho_0 \nu$, there is a strict classical supersolution $-\hat \vp(\cdot - \zeta, \cdot - t_0)$, defined on the set
\begin{align*}
K = \set{\rho_0 - \e \leq \abs{x - \zeta} \leq \rho_0 + \e,\ t \in (t_0 - \e, t_0]},
\end{align*}
for some $\e > 0$.

For $\tau > 0$ and $h = \tau^2$, let us define
\begin{align*}
\Sigma_\tau = \set{\rho_0 < \abs{x- \zeta} - (t- t_0)\hat \om < \rho_0 + h,\ t \in (t_0 - \tau, t_0]},
\end{align*}
see Figure~\ref{fig:Zzero}.
\begin{figure}
\centering
\fig{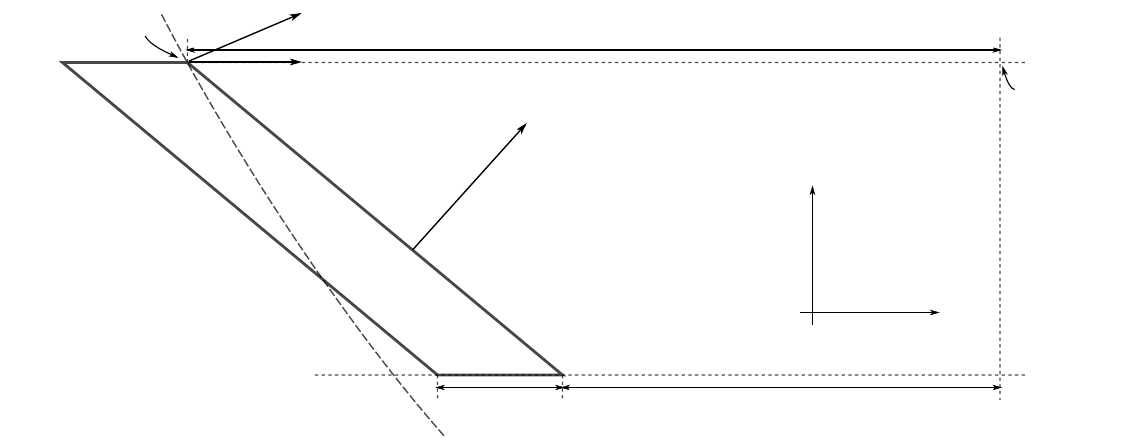}{4.5in}
\caption{Construction of a test function $\vp$ in Lemma~\ref{le:ZWzeroAtContact}}
\label{fig:Zzero}
\end{figure}

We can take $\tau$ small enough so that
\begin{enumerate}
\item $\cl \Sigma_\tau \subset K$,
\item $\cl \Sigma_\tau \cap \set{t = t_0 - \tau} \subset \Xi(P')$.
\end{enumerate}
Now we find $\mu > 0$ large enough so that 
\begin{align*}
-\mu \hat \vp > 2M \quad \text{on} \quad \set{\abs{x} = \rho_0 + h + t \hat \om,\ t \in [-\tau, 0]}.
\end{align*}
For $\eta > 0$ set $\vp_\eta(x,t) = - \mu \hat \vp(x - \zeta - \eta \nu, t - t_0)$.
Since $Z$ is a subsolution of \eqref{RP} and $Z < \vp$ on $\partial_P \Sigma + (\eta \nu, 0)$ for small $\eta$, we also have $Z < \vp$ in $\Sigma + (\eta \nu, 0)$.
We conclude by sending $\eta \to 0$,
\[
Z(\xi, t_0) \leq \lim_{\eta \to 0+} \vp_\eta(\xi, t_0) = \vp_0(\xi, t_0) = 0. \qedhere
\]
\end{proof}

Lemmas~\ref{le:negativeSetsOrdered} and \ref{le:ZWzeroAtContact} have the following consequence:
\begin{corollary}
\label{cor:orderingAtt0}
Let $Z$, $W$ and $t_0$ be as above. Then 
\begin{align*}
Z \leq W \quad \text{at } t = t_0.
\end{align*}
Furthermore,  
\begin{align*}
\set{Z = 0}_{t_0} \cap \set{W = 0}_{t_0} \neq \emptyset,
\end{align*}
\end{corollary}

\begin{proof}
We apply the comparison principle for elliptic equations (Prop.~\ref{pr:ellipticComparison}) and compare $Z(\cdot, t_0)$ and $W(\cdot, t_0)$ on the open set $\set{Z < 0}_{t_0}$.
Indeed, we can set $Z = 0$ on $\partial \set{Z < 0}$ to ensure that $Z\leq 0$.
The modified $Z(\cdot, t_0)$ is clearly in $USC(\cl{\set{Z<0}_{t_0}})$ and a subsolution of the elliptic problem in $\interior \cl{\set{Z<0}_{t_0}}$ ($=\set{Z<0}_{t_0}$ by Lemma~\ref{le:nohair}).
Then $Z \leq W$ on $\partial \pth{\set{Z < 0}_{t_0}}$ due to Lemma~\ref{le:ZWzeroAtContact} and the elliptic comparison applies.

The same can be done using the comparison principle for the parabolic equation (Proposition~\ref{pr:parabolicComparison}) in the parabolic neighborhood $\set{W > 0,\ t \leq t_0}$.

\vspace{10pt}

Finally, if $\set{Z = 0}_{t_0} \cap \set{W = 0}_{t_0} = \emptyset$, then in the view of Lemma~\ref{le:negativeSetsOrdered} and Lemma~\ref{le:ZWzeroAtContact},
\begin{align*}
\set{Z \geq 0,\ t \leq t_0} \cap \set{W \leq 0,\ t \leq t_0} = \emptyset.
\end{align*}
Since these sets are compact, they have positive distance and the comparison in the first part of the proof yields $Z < W$ at $t = t_0$, a contradiction with the definition of $t_0$ and $Z - W \in USC$.
\end{proof}

\subsection{Ordering of gradients at the contact point}
Here we will show that, at a contact point $P$, the ``gradients" of $Z$ (resp. $W$) follow the flux ordering given in Definition~\ref{de:classicalSubsolution}(iv) (resp. its supersolution counterpart). As mentioned in the beginning of section~\ref{sec:comparison}, such ordering would readily yield a contradiction with the fact that $Z$ crosses $W$ from below at $P$.

\begin{lemma}
\label{le:gradientsOrdered}
Let $Z$, $W$ and $t_0$ be as defined above, let \\$P = (\xi, t_0) \in \set{Z \geq 0} \cap \set{W \leq 0}$ be a contact point at time $t_0$, and let $\nu$ be the unique spatial unit normal vector at $P$ obtained in Corollary~\ref{co:unitNormal}. Then
\begin{align*}
 \liminf_{h \searrow0+} \frac{Z(\xi - h \nu, t_0)}{h} + \liminf_{h \searrow 0+} \frac{Z(\xi + h \nu, t_0)}{h} &\geq 0,\\
  \limsup_{h \searrow 0+} \frac{W(\xi - h \nu, t_0)}{h} + \limsup_{h \searrow 0+} \frac{W(\xi + h \nu, t_0)}{h} &\leq 0.
\end{align*}
\end{lemma}

\begin{proof}
We prove the result for $Z$. The proof for $W$ is similar.

Let us again write $\Xi = \Xi_r(P) \equiv \Xi_r(\xi,t_0)$ throughout the proof. Denote 
\begin{align*}
a &:= \liminf_{h \searrow 0+} \frac{Z(\xi - h \nu, t_0)}{h}, & b &:= \liminf_{h \searrow 0+} \frac{Z(\xi + h \nu, t_0)}{h}.
\end{align*}
Note that $b \leq 0 \leq a$. A simple barrier argument, following the one in the proof of Lemma~\ref{le:ZWzeroAtContact} and taking advantage of the room provided by Lemma~\ref{le:finiteSpeed}, shows that $a < \infty$. Furthermore, the inequality $a + b \geq 0$ is satisfied trivially if $b = 0$. Thus we may replace $b$ by $-2a$ if necessary, and assume that $-\infty < b < 0$.

Suppose that the result does not hold, in other words, $a + b < 0$. In that case, we set $\kappa := \frac{a - b}{2} > 0$, and we note that
\begin{align*}
a &< \kappa, & b &< -\kappa,
\end{align*}
Let us choose $\eta > 0$ such that 
\begin{align*}
a + 3\eta &< \kappa, & b + 3\eta & < - \kappa. 
\end{align*}
We shall use this to construct a barrier that crosses $u$ from above at an arbitrary fixed point $P_u = (\xi', \si') \in \Pi^u(P)$, yielding a contradiction.
As in the proof of Lemma~\ref{le:ZWzeroAtContact}, since $P_u \in \partial_L \Xi$ thanks to Lemma~\ref{le:finiteSpeed}, there is $m_u \in \R$ such that $(\nu, m_u) \in \Rn \times \R$ is an interior normal vector to $\partial \Xi$ at $P_u$ (see Figure~\ref{fig:contact}).

\vspace{10pt}

Following the proof of Lemma~\ref{le:finiteSpeed}, we can construct a radial supersolution of \eqref{RP} with the help of Proposition~\ref{pr:radialSolution}.
Let us again choose $\rho_0$ small so that for $\hat a = a + 3 \eta$ and $\hat b = b + 3\eta$, $\hat \om = \max (0, 2m)$, $\zeta = \xi' + \rho_0 \nu$, and $\hat \si = \si'$, there is $\vp = -\hat \vp$, a strict classical supersolution of \eqref{RP} translated by $(\zeta, \hat \si)$, defined on $K$ as in \eqref{cylinder-K}, and $\vp = 0$ at $P_u$. 

\vspace{10pt}

From the definition of $a$ and $b$, and $Z$, for every $\tau \in (0,1)$ there are $h_1, h_2 \in (0, \tau^2)$ such that
\begin{align}
\label{u-less-than-a+eta}
\sup_{\cl\Xi_r(\xi - h_1 \nu, t_0)} u = Z(\xi - h_1 \nu, t_0) &<  (a + \eta) h_1, \\
\label{u-less-than-b+eta}
\sup_{\cl\Xi_r(\xi + h_2 \nu, t_0)} u = Z(\xi + h_2 \nu, t_0) &<  (b+\eta) h_2.
\end{align}
For a given choice of $\tau, h_1, h_2$, let us define
\begin{align*}
\Sigma_\tau := \set{\rho_0 - h_2 < \abs{x - \zeta} - (t - \hat \si) \hat \om < \rho_0 + h_1,\ t \in (\hat \si - \tau, \hat \si]},
\end{align*}
see Figure~\ref{fig:finspeedv} for a sketch of a similar construction.
We shall choose $\tau \in (0,1)$ small enough so that the following holds:
\begin{enumerate}
\item $\Sigma_\tau \subset K$,
\item $\vp > (a+\eta) h_1$ on $A := \set{\abs{x - \zeta} - (t - \hat \si)\hat \om = \rho_0 + h_1, \ t \in [\hat \si - \tau, \hat \si]}$,
\item $\vp > (b+\eta) h_2$ on $\partial_P \Sigma_\tau \setminus A$,
\item $A \subset \cl{\Xi}_r (\xi - h_1\nu, t_0)$,
\item $\partial_P \Sigma_\tau \setminus A \subset \cl{\Xi}_r (\xi + h_2\nu, t_0)$.
\end{enumerate}
Indeed, $\vp$ satisfies (b) and (c) for small $\tau$ due to its smoothness in the positive and negative phases, and (a), (d) and (e) are a consequence of the choice of $\hat \omega$ and the definition of $\Sigma_\tau$. Since $\vp > u$ on $\partial_P \Sigma_\tau$ (thanks to (b)+(d)+\eqref{u-less-than-a+eta} and (c)+(e)+\eqref{u-less-than-b+eta}) while $0 = \vp  = u$ at $P_u \in \Sigma_\tau$, we get a contradiction.
\end{proof}

\subsection{Proof of Theorem \ref{th:comparisonPrinciple}}
\label{sec:proofOfComparison}

Now we are ready to prove Theorem~\ref{th:comparisonPrinciple}. The proof proceeds by showing the comparison for the regularizations $Z$ and $W$ defined in \eqref{eq:regularizedSolutions}. We choose $r > 0$ such that $Z < W$ on $\partial_P Q_r$ (Proposition~\ref{pr:ZWorderedOnBoundary}).

\vspace{10pt}

Suppose that there is a point in $Q$ where $u \geq v$. Since $u \leq Z$ and $W \leq v$, we see that there must be a point in $Q_r$ where $Z \geq W$. Then $t_0$ in \eqref{eq:crossingTime} is finite. We recall that $Z \leq W$ at $t = t_0$ by Corollary~\ref{cor:orderingAtt0}.

\vspace{10pt}

Corollary~\ref{cor:orderingAtt0} also guarantees the existence of a contact point at $t=t_0$:
$$
P = (\xi, t_0) \in \partial \set{Z \geq 0} \cap \set{W \leq 0}\hbox{ with }Z(P) = W(P) = 0.
$$ At the point $P$, the boundaries $\partial \set{Z \geq 0}$ and $\partial \set{W \leq 0}$ are $C^{1,1}$, in the sense that they have space-time and space balls from both sides (Proposition~\ref{pr:dualPoint}(iii)). Let $\nu$ be the unique unit normal vector to the space balls at $\xi$ given by Corollary~\ref{co:unitNormal}.

\vspace{10pt}

Set $\Omega'$ to be the connected component of the open set $\set{Z < 0}_{t_0}$ that contains $\xi$ on its boundary. Lemma~\ref{le:trunk} guarantees that 
$$
Z(\cdot, t_0) \not\equiv W(\cdot, t_0)\hbox{ on }\Omega'.
$$ We recall that $Z(\cdot, t_0)$ is a subsolution of the elliptic problem on $\Omega'$ and $-W^-(\cdot, t_0)$ is a supersolution of the elliptic problem (Lemma~\ref{le:positiveNegativePart}). We can apply the elliptic Hopf's lemma (Proposition~\ref{pr:Hopf}) to $Z(\cdot, t_0)$ and $W(\cdot, t_0)$ on $\Omega'$ at $\xi$, which yields
\begin{align*}
\liminf_{h \searrow 0+} \frac{- Z(\xi + h\nu, t_0)}{h} > \liminf_{h \searrow 0+} \frac{- W(\xi + h\nu, t_0)}{h}.
\end{align*}

\vspace{10pt}

Finally, the ``weak gradients'' are ordered at $\xi$ by Lemma~\ref{le:gradientsOrdered} which leads to a contradiction (we write $Z(\cdot)$ instead of $Z(\cdot, t_0)$, all limits are $h \searrow 0+$):
\begin{align*}
\liminf \frac{Z(\xi - h \nu)}{h} &\geq  \limsup  \frac{-Z(\xi + h \nu)}{h}  \geq \liminf  \frac{-Z(\xi + h \nu)}{h} \\&\stackrel{\text{Hopf}}{>} \liminf  \frac{-W(\xi + h \nu)}{h} \geq  \limsup \frac{W(\xi - h \nu)}{h} \\&\geq \liminf \frac{Z(\xi - h\nu)}{h}.
\end{align*}
Therefore $t_0$ cannot be finite and we conclude that $u < v$ in $Q$. This finishes the proof of Theorem~\ref{th:comparisonPrinciple}. \qed

\subsection{Remarks on the proof of Theorem~\ref{th:comparisonPrinciple} for divergence-form operators}

The arguments in the proof of Theorem~\ref{th:comparisonPrinciple} also hold for the divergence-form operator $F$ given by \eqref{divergence}.
The only difference lies in the references on the properties of solutions of the parabolic and elliptic problems, which have been used throughout section~\ref{sec:comparison} and which are proved or referred to in Appendix~\ref{ap:elliptic-and-parabolic-theory}, as well as the specific barriers constructed in Appendix~\ref{ap:construction-of-barriers-fully-nonlinear}.

Since the barriers for the divergence-form operator are constructed in Appendix~\ref{ap:construction-of-barriers-divergence-form}, here we only point out the references for the regularity properties of solutions of 
the parabolic and elliptic problems (in the sense of Definition~\ref{def:parabolicElliptic}).
More precisely, we used
the parabolic Harnack inequality (Proposition~\ref{pr:Harnack}) and the elliptic Hopf's lemma (Proposition~\ref{pr:Hopf}). For a divergence-form operator $F$ given by \eqref{divergence}, Proposition~\ref{pr:Harnack} is shown in \cite{LSU}*{Chapter V} and Proposition~\ref{pr:Hopf} is shown in \cite{DDS}.

\vspace{10pt}

Having all of the above properties and test functions, one can proceed as in section~\ref{sec:comparison} to prove the following:

\begin{theorem}~\label{CP:div}
Theorem~\ref{th:comparisonPrinciple} holds for $F$ given in \eqref{divergence}.
\end{theorem}

\section{Existence and stability of viscosity solutions}
\label{sec:existenceAndStability}

Let $F$ and $\Omega$ be given as before. In this section we show the existence of viscosity solutions of \eqref{RP} via the Perron's method. We  say that a continuous function $u_0:\Omega\to \R$ is in $\mathcal{P}$ if the following conditions hold:
\begin{equation}\label{A}
u_0=-1\hbox{ on }\partial\Omega\quad \text{and} \quad-F(D^2u_0,Du_0,u_0)=0\hbox{ in }\{u_0<0\}.
\end{equation}
\begin{equation}\label{B}
\Gamma(u_0):=\partial\{u_0>0\}=\partial\{u_0<0\} \text{ is locally a $C^{1,1}$-graph.}
\end{equation}
Note that \eqref{B} and the boundedness of $\Omega$ guarantee that there exists $R_0 > 0$ such that $\partial \set{u_0 > 0}$ has both an interior and an exterior ball of radius $R_0$ at each point. 

 First we ensure, by constructing suitable barriers, that the solutions with initial data $u_0 \in \mathcal{P}$ evolve continuously at the initial time.
\begin{lemma}\label{stability:initial}
Suppose $u_0\in\mathcal{P}$. Then there exist at least one viscosity subsolution $U$ and one supersolution $V$ of \eqref{RP}  with initial data $u_0$. Moreover, there exists a constant $C>0$ and a small time $t_0>0$ depending only on the maximum of $u_0$, $N$, $d$ and $R_0$ such that for $0<t<t_0$ we have
$$
d(x, \Gamma), d (y,\Gamma) \leq t^{1/4} \hbox{ for } x\in\partial\{U(\cdot,t)>0\} \hbox{ and } y\in \partial\{V(\cdot,t)<0\}.
$$
\end{lemma}

\begin{proof}

1. First we construct a supersolution $V$. Let us set 
$$
\mathcal{O}^+(t):= \{x: d(x, \{u_0> 0\})< t^{1/4}\},
$$
and let $V(x,t)$ solve 
$$
\left\{\begin{array}{lll}
V_t - F(D^2 V, DV, V)>0 &\hbox{ in }&\mathcal{O}^+(t)\\ \\
V=0&\hbox{ on } &\Gamma(t):=\partial\mathcal{O}^+(t)\\ \\
-F(D^2V,DV,V)>0 &\hbox{ in } & \mathcal{O}^-(t):= \Omega \setminus \mathcal{O}^+(t)
\end{array}\right.
$$
with the initial and lateral boundary data $u_0$ and $-1$, respectively. Note that, due to the regularity assumption \eqref{B} on $\Gamma$, the positive phase $\mathcal{O}^+(t)$ has the exterior ball property with radius $R_0-t^{1/4}$ for small time $t$. In particular, it follows from the uniform ellipticity of the operator in the negative phase that  there exists $t_0>0$ such that  we have $|DV| >c_0$ on $\Gamma(t)$ for $0<t<t_0$, where $c_0 > 0$ is independent of $t$.

\vspace{10pt}

Now a barrier argument, which takes advantage of the barrier constructed in the proof of Lemma~\ref{le:Dsubsolution0}, yields the existence of a function $c(t)$ such that
$$
|DV^+| \leq c(t)\hbox{ with }  c(t)\to 0\hbox{ as }t\to\infty.
$$ 
  Hence we can choose $t_0$ sufficiently small such that $|DV^+| < |DV^-| $ on $\Gamma(t)$ for $0\leq t \leq t_0$, and therefore $V$ is a viscosity supersolution of \eqref{RP} for $0\leq t\leq t_0$. For $t> t_0$ we take $V(\cdot, t) = \psi$, where $\psi(x)$ is a viscosity supersolution of the elliptic problem $-F(D^2 \psi, D\psi, \psi) =0$ in $\Omega$ with the boundary data $-1$, and $\psi > V(\cdot, s)$ for $0 \leq s \leq t_0$.

\vspace{10pt}

2. The construction of a subsolution $U$ is similar: we replace $\mathcal{O}^+(t)$ in the construction of $V$ by 
$$
\mathcal{O}^+(t):=\{x: d(x,\{u_0 < 0\}) > t^{1/4}\}.
$$

It follows from a barrier argument in the proof of Lemma~\ref{le:DsupersolutionInfty}  that 
$$
|DU^+| \geq C(t)\hbox{ on }\Gamma(t),\hbox{ with } C(t)\to\infty\hbox{ as }t\to 0.
$$ 
On the other hand due to the exterior ball property $|DU^-| \leq C_0$ on $\Gamma(t)$, where $C_0$ is independent of time.  Hence we can choose $t_0$ sufficiently small such that $|DU^+| > |DU^-| $ on $\Gamma(t)$, and thus $U$ is a viscosity subsolution of \eqref{RP} for $0\leq t\leq t_0$. For $t>t_0$ we take $\mathcal{O}^+(t) := \emptyset$.
\end{proof}

Using $V$ and $U$ constructed above, the minimal solution $\underline{u}$ can be constructed as 
$$
\underline{u} = \inf\{ z: \hbox{viscosity supersolution of \eqref{RP} such that }U \leq z \},
$$
and the maximal solution $\overline{u}$ can be constructed as 
$$
\overline{u}=\sup\{w: \hbox{viscosity subsolution of \eqref{RP} such that } w\leq V\}.
$$

\begin{theorem}\label{existence}
$\overline{u}$ and $\underline{u}$ are viscosity solutions of \eqref{RP} with initial data $u_0$, and $\overline{u} = \overline{u}^*$, $\underline{u} = \underline{u}_*$. Moreover, if $v$ is a viscosity solution of \eqref{RP} with initial data $u_0$ then $\overline{u} \leq v\leq\ \underline{u}$.
\end{theorem}
\begin{proof}

 The proof follows from standard arguments in Perron's method (see for example the proof of Theorem 1.2 in \cite{K}. Also see \cite{CIL}*{section 4}). 
\end{proof}

 Lastly we show stability properties of the maximal (and minimal) solutions under perturbation of initial data.
 
 For a family of functions $\{u^\e\}_{\e>0}$, let us define the half-relaxed limits
\begin{align*}
\limhalfsup_{\e\to 0} u^\e(x,t) := \limsup_{\substack{\e \to 0\\(y,s)\to(x,t)}} u^\e(y,s)
\end{align*}
and
\begin{align*}
\limhalfinf_{\e\to 0} u^\e(x,t) := \liminf_{\substack{\e \to 0\\(y,s)\to(x,t)}} u^\e(y,s).
\end{align*}

\begin{theorem}\label{stability:visc}
 Let $\overline{u}$ be the maximal viscosity solution of \eqref{RP} with initial data $u_0$. Let $u_0^\e$ be a sequence of functions decreasing in $\e \to 0$ such that 
\begin{enumerate}
\romanlist
\item $u_0 < u_0^\e$ and $u_0^\e$ converges uniformly to $u_0$ as $\e\to 0$, 
\item  $u_0^\e$ and  $\Gamma^\e :=\partial\{u_0^\e>0\}=\partial\{u_0^\e<0\}$ satisfy Assumptions \eqref{A} and \eqref{B} uniformly in $\e>0$,
\item  $\Gamma^\e$ converges uniformly to $\Gamma(u_0)$ with respect to the Hausdorff distance.
 \end{enumerate}
 Let $u^\e$ be a viscosity solution with initial data $u_0^\e$. Then $u^\e$ converges to $\overline{u}$ in the following sense: both $U_1:=\limhalfsup_{\e\to 0}u^\e$ and $U_2:=\limhalfinf_{\e\to 0} u^\e$ satisfy
$$
(U_i)^* = \overline{u} \hbox{ and } (U_i)_*=\overline{u}_*.
$$
Corresponding results hold for the minimal solutions, if one replaces the condition $u_0< u_0^\e$ in (i) with $u_0^\e < u_0$, and $u_0^\e$ is increasing as $\e \to 0$.
\end{theorem}

\begin{remark}
Note that the convergence in the sense of the above theorem is optimal for semi-continuous functions.
\end{remark}

\begin{proof}
1.  First observe that Theorem~\ref{th:comparisonPrinciple} yields that 
 \begin{equation}\label{one}
 \overline{u}^*< (u^\e)_*\hbox{ for any } \e>0.
 \end{equation}
On the other hand, due to the standard stability property of viscosity solutions, \\$\limhalfsup_{\e\to 0}u^\e(y,s)$
  is a subsolution of \eqref{RP}. Hence by a barrier argument (one may use the supersolution $V^\e$ constructed as $V$ in the proof of Lemma~\ref{stability:initial} but for the initial data $u_0^\e$) we can show that 
 \begin{equation}
 \limhalfsup_{\e\to 0} u^\e \leq V.
 \end{equation}
  Therefore by definition of $\overline{u}$ we have 
 \begin{equation}\label{three}
  \limhalfsup_{\e\to 0} u^\e \leq \overline{u}.
  \end{equation}
    Putting \eqref{one} and \eqref{three} together, we conclude.
\end{proof}

From the above theorem we have the following ``regularity" information on the minimal and maximal viscosity solutions.

\begin{corollary}
\begin{equation}\label{equality}
(\overline{u}_*)^* = \overline{u} \hbox{ and } (\underline{u}^*)_*= \underline{u}.
\end{equation}
\end{corollary}

\section{Uniqueness properties}
\label{sec:uniqueness}

Recall that our comparison principle, Theorem~\ref{th:comparisonPrinciple}, requires strictly separated initial data. We saw in the previous section that the theorem yields existence (and uniqueness, by definition) of maximal and minimal viscosity solutions. In this section we will discuss uniqueness properties of general viscosity solutions. 

\subsection{Viscosity solutions coincide with regular weak solutions}

In this section we consider a linear uniformly elliptic operator $F$ of the form $F(M, p, z) = \trace M$.
We show that in this case the viscosity solutions of \eqref{RP} are precisely the regular weak solutions introduced in \cite{AL}.
As a corollary it follows that viscosity solutions are unique for this class of operators.

For completeness of the paper we revisit the notion of regular weak solutions and its properties.
We refer to Section~1.4 and Definition~2.1 of \cite{AL} for the definition of weak solutions, subsolutions and supersolutions of \eqref{RP}, which are as usual defined in $L^2(0,T;H^1(\Omega))$ via integration by parts.

\begin{theorem}[\cite{AL}, Theorem 2.2: simplified version]\label{thm:cp}
Suppose that
\begin{align*}
F(M,p,z) = \trace M,
\end{align*}
and let $Q = \Omega \times (0, T]$, where $\Omega$ is a bounded domain with Lipschitz boundary.
Suppose that $u_1$ is a weak subsolution and $u_2$ a weak supersolution of \eqref{RP}, with initial data $b^0_1$, $b^0_2$, and boundary data $u^D_1$, $u^D_2$, respectively, such that 
\begin{equation}\label{reg}
\partial_t b(u_1), \quad \partial_t b(u_2) \in L^2(Q).
\end{equation}
If $b^0_1 \leq b^0_2$ a.e. in $\Omega$ and $u^D_1 \leq u^D_2$ a.e. on $\partial_L Q$ then $u_1 \leq u_2$ a.e. in $Q$.
\end{theorem}
 
 We call $u$ a {\it regular weak solution} if $u$ is a weak solution and satisfies the additional regularity \eqref{reg}. It is shown in \cite{AL} (Theorem 2.3) that for linear $F$ there exists a regular weak solution with initial data $u_0\in \mathcal{P}$.
 
 \vspace{10pt}
 
  Observing that the test functions $U$ and $V$ constructed in the proof of Lemma~\ref{stability:initial} are respectively a classical subsolution and a classical supersolution of \eqref{RP} for $0<t<t_0$, Theorem~\ref{thm:cp} yields the following:
 
 \begin{lemma}
 Let $u_0\in H^1(\Omega)$ satisfy assumptions \eqref{A}-\eqref{B}. Then any regular weak solution  $u$ of \eqref{RP} with initial data $u_0$ uniformly converges to $u_0$ as $t\to0$.
 \end{lemma} 
 
 Furthermore, the following holds due to the estimates derived in the proof of Theorem 2.3. in \cite{AL}:
\begin{lemma}\label{stability:weak}
 Regular weak solutions are stable under perturbation of initial data in $H^1(\Omega)$. More precisely if the sequence of initial data $u^0_n$ converges to $u_0$ in $H^1(\Omega)$, then the corresponding regular weak solutions $u_n$ with respect to the initial data $u^0_n$ converges to the regular weak solution of $u$ with initial data $u_0$  in $L^2(0,T; H^1(\Omega))$.
 \end{lemma}
 
Given a weak solution $u \in L^\infty(Q)$, we want to find a representative $v = u$ a.e. that is a suitable candidate for a viscosity solution.
We shall find it using a weaker notion of semi-continuous envelopes.
To this end, let us state an elementary lemma.

\begin{lemma}
\label{leOrderOfEnvelopes}
Let $E \subset \Rn$ such that $\abs{E \cap B_\de(x)} > 0$ for all $x\in E$ and $\de > 0$. If $f\leq g$ a.e. in $E$, then
\begin{align*}
 f_{*,E} \leq g^{*,E}.
\end{align*}
\end{lemma}

Given a function $u \in L^\infty(Q)$, where $Q = \Omega \times (0,T]$, define the \emph{essential semi-continuous envelopes}
\begin{align*}
u^\diamond(x, t) &= \inf_{v = u \text{ a.e.}} v^{*,\cl Q}(x,t) = \inf_{r>0} \esssup_{B_r(x,t) \cap \cl Q} u ,\\
u_\diamond(x,t) &= \sup_{v = u \text{ a.e.}} v_{*,\cl Q}(x,t) = \sup_{r > 0} \essinf_{B_r(x,t) \cap \cl Q} u.\\
\end{align*}
Here $v^{*,\cl Q}$ and $v_{*, \cl Q}$ are as defined in \eqref{usclsc}.

It follows from the definition that $u^\diamond \in USC(\cl Q)$ and $u_\diamond \in LSC(\cl Q)$, and due to Lemma \ref{leOrderOfEnvelopes} we have 
$$
u_\diamond \leq u^\diamond.
$$ We claim that $u_\diamond \leq u \leq u^\diamond$ a.e. To see this, we have for almost every $P = (x,t)$ by the Lebesgue's differentiation theorem
\begin{align*}
u^\diamond(P) &= \inf_{v = u \text{ a.e.}} v^*(P) = \inf_{v = u \text{ a.e.}} \inf_{r>0} \sup_{B_r(P)} v\\
&= \inf_{r>0} \inf_{v = u \text{ a.e.}} \sup_{B_r(P)} v \geq \lim_{r \to 0} \frac{1}{\abs{B_r(P)}} \int_{B_r(P)} u = u(P), \quad \text{a.e. } P.
\end{align*}

Next we introduce a candidate function for the viscosity solution.

\begin{lemma}
\label{le:viscosityCandidate}
Given a function $u \in L^\infty(Q)$, define
\begin{align*}
v = \max(\min(u, u^\diamond),u_\diamond).
\end{align*}
Then $v = u$ a.e., $v_* = u_\diamond$ and $v^* = u^\diamond$.
\end{lemma}

\begin{proof}
Observe that $u_\diamond \leq v \leq u^\diamond$ since $u_\diamond \leq u^\diamond$.
In fact, since $u_\diamond \leq u \leq u^\diamond$ a.e we also have that $v = u$ a.e.
Now $u^\diamond \in USC$ implies $v^* \leq u^\diamond$ and the definition of $u^\diamond$ as the infimum yields $u^\diamond \leq v^*$. Similarly $u_\diamond = v_*$.  
\end{proof}

Now we are ready to state the main result in this subsection:

\begin{theorem}\label{coincidence}
Let $\mathcal{P}$ be the class of regular initial data as defined in Section~\ref{sec:existenceAndStability} and let $u$ be the unique regular weak solution with initial data $u_0 \in \mathcal{P}$.
Then $v$ as given in Lemma~\ref{le:viscosityCandidate} is a viscosity solution of \eqref{RP} for the initial data $u_0$.
\end{theorem}

\begin{proof}
The proof is based on the local comparison principle that regular weak solutions satisfy (Theorem~\ref{thm:cp}). First note that $v$ uniformly converges to $u_0$ at $t=0$ due to Theorem~\ref{thm:cp}, and the barrier arguments using the barriers $V$ and $U$ constructed in Lemma~\ref{stability:initial}. 

\vspace{10pt}

We only show that $v_*$ is a viscosity supersolution; parallel arguments apply to the subsolution part.
Let $\phi$ be a strict classical subsolution in a cylindrical domain $Q':= \Omega' \times (t_1,t_2] \subset Q$ such that $\phi < v_*$ on $\partial_P Q'$.

Our goal is to show that $\phi < v_*$ in $Q'$. Thus assume that this does not hold and there exists $(x_0, t_0) \in Q'$ with $\phi \geq v_*$ at $(x_0,t_0)$.

We claim that we can perturb $\phi$ into a strict classical subsolution $\tilde{\phi}$ with $\tilde{\phi} > \phi$ at $(x_0,t_0)$ and $\tilde \phi < v_*$ on $\partial_P Q'$. 
To this end, let $\ta \in C^\infty_c(\Rn)$, $\ta \geq 0$, be the standard smooth mollifier with support $\cl{B}_1^n$. For $\e, \eta> 0$ let $\ta_{\e, \eta}(x) : = \e \ta(\eta^{-1} x)$ and define $\tilde \phi$ as
\begin{itemize}
\item if $\phi(x_0,t_0) = 0$: $\tilde \phi(x,t) = \phi(x + \ta_{\e,\eta}(x) \nu, t)$, where $\nu$ is the unit outer normal of $\set{\phi > 0}_{t_0}$ at $x_0$,
\item if $\phi(x_0,t_0) \neq 0$: $\tilde \phi(x,t) = \phi(x,t) + \ta_{\e,\eta}(x- x_0)$.
\end{itemize}
If $\e$ and $\eta$ are chosen small enough, a straightforward differentiation shows that the perturbation $\tilde \phi$ has the required properties.

There is also $\de, \eta > 0$ such that $\tilde \phi + \de < v_*$ on an $\eta$-neighborhood of $\partial_P Q'$, i.e. on the set
\begin{align*}
\set{(x,t) \in \cl Q': \dist((x,t), \partial_P Q') < \eta}.
\end{align*}
Finally, $v_* \leq u$ a.e. in $Q'$ by definition.
We conclude that:
\begin{enumerate}
\romanlist
\item $\tilde \phi(\cdot, t) \leq u(\cdot, t)$ on $\partial \Omega'$ for a.e. $t \in (t_1,t_2]$ in the sense of trace on $H^1(\Omega')$;
\item due to regularity \eqref{reg}, $b(u) \in C(t_1, t_2; L^2(\Omega'))$  (see \cite{Evans}*{\S5.9.2 Theorem 3} or \cite{GGZ}*{Chapter IV, Theorem 1.17}) and therefore there is a unique $b_0' \in L^2(\Omega')$ such that $b(u(t)) \to b_0'$ strongly in $L^2(\Omega')$ as $t \to t_1$, and $\tilde \phi(\cdot, t_1) \leq b_0'$ a.e. on $\Omega'$.
\end{enumerate}
A simple computation shows that $\tilde \phi$ is a regular weak subsolution in the sense of \cite{AL}*{Definition 2.1}.
Therefore, by Theorem~\ref{thm:cp}, we have $\tilde \phi \leq u$ a.e. in $Q'$ and hence $\phi < \tilde \phi = \tilde \phi_\diamond \leq u_\diamond = v_*$ at $(x_0,t_0)$, a contradiction. Therefore $v_*$ is a viscosity supersolution of \eqref{RP} due to Remark~\ref{re:only-cylinders}.
\end{proof}

Theorem~\ref{stability:visc} and Lemma~\ref{stability:weak} yield the following:
 
\begin{corollary}\label{coro:unique}
Let $F(M, p, z) = \trace M$ and let $u_0\in \mathcal{P}$. Then there exists a unique viscosity solution $v$ of \eqref{RP} with initial data $u_0$, which coincides almost everywhere in $Q = \Omega \times (0,\infty]$ with the regular weak solution with initial data $u_0$.
\end{corollary}

\subsection{Discussion of uniqueness for nonlinear operators}

At the moment, the uniqueness results for fully nonlinear operators are unfortunately quite limited to initial data and domains under special geometries. Here we illustrate an example.  The authors suspect that uniqueness holds in most settings, but it remains an open question.  

When
\begin{equation}\label{simple}
b(u)=u_+,\quad F = F(M) \hbox{ with } F(\mu M) = \mu F(M) \text{ for } \mu >0,
\end{equation}
 the following holds.

Let $\Gamma_0:=\partial\{u_0>0\}=\partial\{u_0<0\}$ be a locally $C^{1,1}$ graph, and suppose the initial data $u_0$ satisfies one of the following:
\begin{itemize}
\item[(A)] (Initially decreasing: only possible) $u_0$ is $C^1$ in $\overline{\{u_0>0\}}$ with
$$
-F(D^2 u_0)>0\hbox{ in }\{u_0>0\}\hbox{ and }|Du_0^+|< |Du_0^-|\hbox{ on }\Gamma_0;
$$
\item[(B)] (Star-shaped) $g=-1$, $\Omega$ is a star-shaped domain, and  $u_0((1+\e)x) < u_0(x)$ for all $x\in (1+\e)^{-1}\Omega$ except at $x= 0$ for any $\e>0$.
\end{itemize}

\begin{theorem}\label{uniqueness}
Suppose that $(A)$ or $(B)$ holds. Then for the operator of the form \eqref{simple} there exists a unique viscosity solution of \eqref{RP} with initial data $u_0$. 
\end{theorem}

\begin{proof}
1. In the case of (A), let us define $u^\e_0(x):= (u_0-\e)_+$ in $\{u_0>\e\}$ and $u^\e_0$ be the solution of 
$$
-F(D^2 u)=0 \hbox{ in } \{u_0<\e\}
$$
with the boundary data $g$. Since $u_0$ is $C^1$ and $D(u_0^-)$ changes continuously with respect to the change of the domain (due to the ellipticity of $F$ and the regularity of the domain), if $\e$ is sufficiently small we have $|D(u^\e_0)^+| < |D(u_0)^-|$.

\vspace{10pt}

 Now let us define $V(x,t):= u^{t}_0(x)$. Then for small time, $V(x,t)$ is a viscosity supersolution of \eqref{RP}. It follows then that, for sufficiently small $\e>0$,
\begin{equation}\label{dec}
u(x,\e) < u(x,0) \hbox{ in } \Omega,
\end{equation}
and there exists a constant $c_\e \to 0^+$ as $\e\to 0$ such that 
$(1+c_\e)u(x,\e) < u(x,0)$ in $\cl{\Omega}.$
 Let $u$ and $v$ be two viscosity solutions of \eqref{RP} with initial data $u_0$. In case of (A) we perturb by $(1+c_0\e)u(x,t+\e)$ for given $\e>0$ and apply \eqref{dec} as well as Theorem~\ref{th:comparisonPrinciple} to show that $ u\leq v$. 
 
 \vspace{10pt}
 
 2. In case of (B), first note that from the star-shapedness assumption it follows that $u_0(0)>0$. 
 Hence for given $\e>0$ 
 $$
 (1+c_\e)u((1+\e)x, (1+\e)^2 t)\leq v\hbox{ in }(1+\e)^{-1}\Omega \times [0,\infty).
 $$
 Since $\e$ is arbitrary, we conclude that $u \leq v$.
\end{proof}

\section{Approximation by uniformly parabolic problems} 
\label{sec:approximation}
In this section we show that our problem \eqref{RP} can be obtained as a singular limit of uniformly parabolic problems.  
Consider $b_n \in C^2(\R)$ to be a smooth approximation of $b(s)=s_+$ with the following properties:
\begin{enumerate}
\item $0 < b_n' < 1$ on $\R$,
\item $b_n \to b$ locally uniformly on $\R$,
\item $b'_n \to 0$ locally uniformly on $(-\infty, 0]$,
\item $b'_n \to 1$ locally uniformly on $(0, \infty)$.
\end{enumerate}

For example, one can choose
$$
b_n(s) = \frac{1}{n^2} \log \frac{e^n + e^{n^2s}}{e^{n} + 1}.
$$

\begin{remark}
The interval in condition (c) includes $0$. This moves the irregularity of $b'_n$ into the positive phase and thus simplifies the proof of Proposition~\ref{limit:visc}, but it is not an essential assumption.
\end{remark}

For $n \geq 1$, let $u_n$ be the unique solution of the smooth parabolic problem
\begin{equation}\label{approx}
\begin{cases}
b_n(u_n)_t - F(D^2 u_n, D u_n, u_n) = 0, & \text{in } Q := \Omega \times (0,\infty),\\
u_n = g&\text{on $\partial_P Q$}.
\end{cases}
\end{equation}
Here we assume that $g$ is a bounded Lipschitz continuous function on $\partial_P Q$.

Due to the maximum principle, $u_n$'s are bounded uniformly in $n$ and therefore  problem \eqref{approx} is uniformly parabolic for each $n$.
Note that $u_n$ is $C^{1,\alpha}$ in $Q$ for each $n$ due to \cite{LSU} (for $F$ given as in \eqref{divergence}) and \cite{W} (for nonlinear $F$).
Also note that one cannot directly use the results from \cite{CIL} since the operator does not converge uniformly. Indeed, $b'$ is discontinuous in our setting. 

 Define the half-relaxed limits
\begin{align*}
\overline{\omega} := \limhalfsup_{n\to\infty} u_n, \quad \underline{\omega} := \limhalfinf_{n\to\infty} u_n.
\end{align*}

\begin{proposition}\label{limit:visc}
$\overline{\omega}$ and $\underline{\omega}$ are respectively a viscosity subsolution and a viscosity supersolution of \eqref{RP} in $Q$.
\end{proposition}

\begin{proof}
Let us denote $u = \overline{\omega}$. The goal is to show that $u$ is a viscosity subsolution of \eqref{RP}. 

1. Suppose not. By definition, there exists $\delta>0$ and a strict classical supersolution $\phi$ of \eqref{RP} in a parabolic neighborhood $E = U \cap \set{t \leq \tau} \subset Q$, satisfying 
\begin{enumerate}
\romanlist
\item $-F(D^2\phi, D\phi,\phi) \geq \delta$ if $ \phi < 0$,
\item $\phi_t-F(D^2\phi, D\phi,\phi) \geq \delta$ if  $\phi > 0$,
\item $|D\phi^+| +3 \delta  <|D\phi^-|$ on $\{\phi=0\}$,
\end{enumerate}
 such that $u < \phi$ on $\partial_P E$ while $\phi \leq u$ at some point of $P \in E$. By increasing $\tau$ if necessary, we may assume that $P \in \set{t < \tau}$.  And by perturbing $\phi$ as in the proof of Theorem~\ref{coincidence}, we may assume that in fact $\phi < u$ at $P \in E$. Therefore for large $n$ we have $u_n < \phi$ on $\partial_P E$ while $u_n > \phi$ at some point of $E$.

\vspace{10pt}

2. For large $n$, due to the continuity of $u_n$, we can find $(x_n,t_n) \in E$ with $u_n- \phi = 0$ at $(x_n, t_n)$ and $u_n -\phi < 0$ in $\set{t < t_n}$. Compactness of $\cl{E}$ allows us to select a subsequence, also denoted $n$, such that $(x_n, t_n) \to (x_0, t_0) \in \cl E$. Since $u < \phi$ on $\partial_P E$ we conclude that $(x_0, t_0) \in E$. Moreover, observe that $\phi(x_0,t_0) = 0$, since otherwise (i)--(ii) and the locally uniform convergence of $b'_n$ to $b'$ on $\R \setminus \set{0}$ yield a contradiction.

The regularity of $u_n$ and the condition (iii) on $\phi$ ensure that $\phi(x_n, t_n) \neq 0$.
Next suppose that along a subsequence we have $\phi(x_n, t_n) < 0$. Since also
$$
\bra{b_n(\phi)_t - F(D^2\phi,D\phi,\phi)}(x_n,t_n) \leq 0,
$$
the uniform convergence of  $b_n' $ to $b'$ in $(-\infty, 0]$ contradicts (i).

\vspace{10pt}

3. Therefore we can assume that $\phi(x_n, t_n) > 0$.  Let us first observe that $\phi_t^+(x_0, t_0) > 0$: this can be verified by adding the two inequalities
 \begin{align*}
\bra{\phi_t - F(D^2 \phi, D\phi,\phi)}(x_n,t_n) &\geq \delta\\
\bra{-b_n'(\phi)\phi_t + F(D^2\phi,D\phi,\phi)}(x_n,t_n) &\geq 0,
\end{align*}
and by using the fact that $b_n'\in(0,1)$. In particular, this implies that $\set{\phi > 0}$ is expanding at $(x_0, t_0)$. 

Let $\nu$ be the unit outer normal of $\partial \set{\phi > 0}_{t_0}$ at $x_0$. Due to the regularity of $\set{\phi = 0}$, it is possible to choose $\rho_0 > 0$ small enough such that $\cl B_{\rho_0}(\zeta_0) \cap \cl{\set{\phi > 0}_{t_0}} = \set{x_0}$, where $\zeta_0 = x_0 + \rho_0 \nu$.
For small enough $\rho_0$, Proposition~\ref{pr:radialSolution} provides a strict classical supersolution $\psi$ of \eqref{RP}, with parameters 
\begin{align*}
\hat a = \abs{D\phi^+(x_0, t_0)} + \de,\quad \hat b = -\abs{D\phi^-(x_0, t_0)} + \de \quad\text{and}\quad \hat \om = 0,
\end{align*}
in a neighborhood $K$ of $\partial B_{\rho_0} \times \set{0}$. Since $\hat \omega = 0$, $\psi$ is independent of time and thus it is a classical strict supersolution of the elliptic problem in each phase.
Let us define $\hat \psi_\eta(x,t) : = \psi(x - \zeta_0 + \eta \nu, t)$.
Choose $\de_1 > 0$ and $\eta > 0$ sufficiently small so that $\hat \psi_\eta(x,t) > \phi(x,t)$ on $\partial_P \Sigma_{\de_1}$ and $\Sigma_{\de_1} \subset K + (\zeta, t_0)$, where
\begin{align*}
\Sigma_{\de_1} = B_{\de_1}(x_0) \times (t_0-\de_1, t_0].
\end{align*}
This is possible due to the gradient ordering condition (iii) of $\phi$ and due to the fact that $\phi_t^+ > 0$ at $(x_0,t_0)$. 

Note that $\hat \psi_\eta < \phi$ at $(x_0,t_0)$. Therefore there is again a subsequence of $n$, and some other sequence $(x_n,t_n) \in \Sigma_{\de_1}$ where $u_n - \hat \psi_\eta$ has a maximum zero at $(x_n,t_n)$ in $\Sigma_{\de_1} \cap \set{t \leq t_n}$. As before, regularity of $u_n$ ensures that $\psi(x_n,t_n) \neq 0$ and so we have
\begin{align*}
\bra{b_n'(\psi) \psi_t - F(D^2 \psi, D\psi,\psi)} (x_n,t_n) \leq 0.
\end{align*}
This contradicts the fact that $\psi$ satisfies $-F(D^2\psi,D\psi,\psi)(x_n,t_n) > 0$.
\end{proof}

\begin{corollary}\label{approx11}(For nonlinear $F$)
Let $u_0\in\mathcal{P}$ and $u_0^\e$ be as given in Theorem~\ref{stability:visc}, and choose a sequence  $\e_k\to 0$. Then for given $T>0$ there exists $n_k\to \infty$ such that $\{u_{n_k}\}$ solving \eqref{approx} with initial data $u_0^{\e_k}$ converges to the maximal viscosity solution $\overline{u}$ of \eqref{RP} with initial data $u_0$. More precisely,  both 
$$U_1:=\limhalfsup_{k\to\infty} u_{n_k}\hbox{ and } U_2:=\limhalfinf_{k\to\infty} u_{n_k}$$ 
 satisfy
$$
(U_i)^* = \overline{u} \hbox{ and } (U_i)_* = \overline{u}_*.
$$
Parallel statements hold for the minimal viscosity solution.
\end{corollary}
\begin{proof}

The main challenge in this proof is to ensure that $\overline{\omega}$ and $\underline{\omega}$ assume the correct boundary and initial data. Let us denote $M:=\max\set{|u_0|,1}$ as the uniform bound on $|u_n|$ obtained by the maximum principle. 

\vspace{10pt}

1. We first construct barriers to control $u_n$ near the lateral boundary data.
Consider $\vp(\cdot, t)$, which solves $F(D^2\vp, D\vp, \vp)=0$ in $\Omega$ for each $t \geq 0$ with boundary data $g(\cdot, t)$. This gives a uniform bound from below on $u_n$ at the lateral boundary $\partial_L Q$. 
Next define $\vp(x,t) = \psi(x)$ where $\psi$ solves $F(D^2\psi, D\psi, \psi)=0$ in $\Sigma_{\delta}:= \{x:d(x,\partial\Omega)<\delta\}$ with boundary data $g$ on $\partial \Omega$ and $M$ on $\{x: d(x,\partial\Omega)=\delta\}.$ Here $\de > 0$ is  chosen small enough so that $\psi \geq g(\cdot, 0)$. Then $u_n \leq \vp$ in the set $\Sigma_{\delta}$ and we obtain the uniform bound from above on $u_n$ at the lateral boundary $\partial_L Q$.

\vspace{10pt}

2. To obtain estimates on $u_n$ near the initial data, one can consider barriers which are defined for short time $0\leq t\leq t_0$ as follows.
For given $\delta>0$ let $V(x,t;\delta)$ be the strict supersolution of \eqref{RP} for $0\leq t\leq t_0$, constructed as in the proof of Lemma~\ref{stability:initial}, with the initial data $u_0^\delta:=u_0+\delta$, boundary data $\delta$ on $\Gamma(t)$ and
\begin{align*}
\mathcal{O}^+(t) = \{x: d(x, \{u_0> 0\})< t^{1/4}\} = \{x: d(x, \{u_0^\de> \de\})< t^{1/4}\}.
\end{align*} 
Note that then we have
$$
\mathcal{O}^+(t)=\{V>\delta\} \hbox{ and } \mathcal{O}^-(t):=\Omega-\mathcal{O}^+(t)=\{V \leq \delta\}.
$$
We observe that $V_t\geq 0$ in $\mathcal{O}^-(t)$ since $V$ solves the elliptic equation with boundary data $\delta$ in $\mathcal{O}^-(t)$ and  the set shrinks in time. Moreover, as mentioned in the proof of Lemma 4.1, due to the uniform ellipticity of the operator in the set $\mathcal{O}^-(t)$, we have 
 \begin{equation}\label{grad_order}
 |DV^+| <|DV^-|\hbox{ on } \Gamma(t).
 \end{equation}
 Now $\vp(x,t):=V(x,t;\delta)$ satisfies
$$
\vp_t-F(D^2\vp,D\vp, \vp) > -F(D^2\vp,D\vp,\vp) > 0 \hbox{ in } \mathcal{O}^-(t)
$$
 and
 $$
 \vp_t-F(D^2\vp,D\vp,\vp) > 0 \hbox{ in } \mathcal{O}^+(t).
 $$
The above inequalities as well as \eqref{grad_order} yield that $\vp$ is a supersolution of \eqref{approx} for $0\leq t\leq t_0$ if $n$ is sufficiently large with respect to $\delta$. Since $\vp(x,0) =u_0+\delta >u_0$,  it follows from the comparison principle for \eqref{approx} that $u_n\leq \vp$  for $0\leq t\leq t_0$ and for sufficiently large $n$. Therefore we have $\overline{\omega}(x,0)\leq u_0(x)+\delta$, and we conclude that $\overline{\omega}(x,0) \leq u_0(x)$.  
 
 To show that $\underline{\omega}(x,0)\geq u_0(x)$ one can argue similarly as above to construct a subsolution of \eqref{approx}, with $U$ in the proof of Lemma~\ref{stability:initial}  and with the initial data $u_0-\delta$.

\vspace{10pt}

3. Now we construct $n_k$ as follows. Note that, for any sequence $n_k\to\infty$, Proposition~\ref{limit:visc} as well as the boundary data estimates given in steps 1.--2. yield that 
 $$
 \limhalfsup_{k\to\infty} u_{n_k} \leq \overline{u}^*.
 $$
 
 For given $\e_k$, choose $n_k$ such that $u_{n_k}$ satisfies 
 $$
\underline{u}_{\e_k} -\e_k \leq u_{n_k}  \hbox{ for } 0\leq t\leq T,
 $$
where $\underline{u}_{\e_k}$ is the minimal viscosity solutions of \eqref{RP} with initial data $u_0^{\e_k}$: such $n_k$ exists due to Proposition~\ref{limit:visc}.  Now we conclude by applying Theorem~\ref{stability:visc} to $\underline{u}_{\e_k}$.

\end{proof}
\begin{corollary}(For nonlinear $F$)
Under the setting of Theorem~\ref{uniqueness}, $\{u_n\}$ converges to $u$ in the sense of Corollary~\ref{approx11}.
\end{corollary}

\begin{corollary}
Suppose $u_0\in \mathcal{P}$ and $F$ is as given in Corollary~\ref{coro:unique}. Then $u_n$ with the initial data $u_0^n$  converges to the unique viscosity solution $u$ of \eqref{RP} with initial data $u_0$ in the sense of Corollary~\ref{approx11}.
\end{corollary}

\section*{Appendix}
\appendix
\section{Elliptic and parabolic theory}
\label{ap:elliptic-and-parabolic-theory}

In this section we collect a few classical results from the theory of fully nonlinear elliptic and parabolic operators used throughout this article.

\begin{proposition}[Elliptic comparison]
\label{pr:ellipticComparison}
Let $\Omega\subset \Rn$ be a bounded domain and let $F$ satisfy \eqref{eq:structural} and \eqref{eq:proper}. Suppose that $u \in USC(\cl{\Omega})$ is a viscosity subsolution and $v \in LSC(\cl{\Omega})$ is a viscosity supersolution of the elliptic problem $-F(D^2 u, Du, u) = 0$ in $\Omega$, and $u \leq v$ on $\partial \Omega$. Then $u \leq v$ in $\Omega$.
\end{proposition}

\begin{proof}
This is a consequence of ABP estimate applied to $u - v$, see \cite{ATh}*{Proposition 2.17}.
\end{proof}

\begin{proposition}[Parabolic comparison]
\label{pr:parabolicComparison}
Let $\Sigma \subset \Rn\times \R$ be a bounded parabolic domain and let $F$ satisfy \eqref{eq:structural} and \eqref{eq:proper}. Suppose that $u \in USC(\cl{\Sigma})$ is a viscosity subsolution and $v \in LSC(\cl{\Sigma})$ is a viscosity supersolution of the parabolic problem $u_t - F(D^2 u, Du, u) = 0$ on $\Sigma$ such that $u \leq v$ on $\partial_P \Sigma$. Then $u \leq v$ in $\Sigma$.
\end{proposition}

\begin{proof}
See for instance \cite{CIL}*{Theorem 8.2} or \cite{WangI}*{Corollary 3.15}.
\end{proof}

\begin{proposition}[Elliptic Hopf's lemma]
\label{pr:Hopf}
Let $u$ and $v$ be respectively a subsolution and a supersolution of $-F(D^2 u, Du, u) = 0$, $u \leq v$ on a domain $\Omega$ and $u \not\equiv v$. If $u = v$ at $x_0 \in \partial \Omega$ and $\Omega$ has a interior ball touching $x_0$, then
\begin{align}
\label{eq:hopfuv}
\liminf_{h\searrow 0+} \frac{u(x_0) - u(x_0 - h \xi)}{h} > \liminf_{h\searrow 0+} \frac{v(x_0) - v(x_0 - h \xi)}{h},
\end{align}
for any $\xi$, $\xi \cdot \nu > 0$, where $\nu$ is the unit outer normal to the interior ball at $x_0$, provided that the right-hand side is finite.
\end{proposition}

\begin{proof}
1. For $\omega = u - v$, we first show that
\begin{align}
\label{eq:hopfomega}
\liminf_{h\searrow 0+} \frac{\omega(x_0) - \omega(x_0 - \xi h)}{h} > 0.
\end{align}
Indeed, \eqref{eq:structural} yields
\begin{align*}
-\mathcal{M}^-(M+N) - \de_1 \abs{p + q} - \de_0 \abs{z + w} &\leq F(-N, -q, -w) - F(M, p, z).
\end{align*}
Since $-v$ is a subsolution of $-G(D^2v, Dv, v) = 0$, where 
\begin{align*}
G(N, q, w) = -F(-N, -q, -w),
\end{align*}
we can use \cite{ATh}*{Proposition 2.14} to conclude that $\om = u + (- v)$ is a subsolution of  
\begin{align*}
-\mathcal{M}^-(D^2\om) - \de_1 \abs{D\om} - \de_0 \abs{\om} =0.
\end{align*}
Thus Hopf's lemma \cite{ATh}*{Proposition 2.15} applied to $\om$ at $x_0$ (recall that $\om(x_0) = 0$) yields \eqref{eq:hopfomega}.

2. Using the property of $\liminf$, we rewrite
\begin{align*}
\liminf_{h\searrow 0+} \frac{u(x_0) - u(x_0 - h \xi)}{h} &\geq \liminf_{h\searrow 0+} \frac{v(x_0) - v(x_0 - h \xi)}{h} \\&+ \liminf_{h\searrow 0+} \frac{\om(x_0) - \om(x_0 - h \xi)}{h},
\end{align*}
which together with \eqref{eq:hopfomega} implies \eqref{eq:hopfuv}.
\end{proof}

\begin{proposition}[Parabolic Harnack's inequality]
\label{pr:Harnack}
Suppose that $F$ satisfies the structural condition \eqref{eq:structural}, is proper \eqref{eq:proper}, and $F(0,0,0) = 0$. Let $t_1$, $t_2$ satisfy $0 < t_1 < t_2$. There exists a constant $c$, depending only on $\la$, $\Lambda$, $\de_1$, $\de_0$, $t_1, t_2$ and $n$, such that 
\begin{align*}
\sup_{D_\e (0, \e^2 t_1)} u \leq c \inf_{D_\e(0,\e^2 t_2)} u
\end{align*}
for any $\e \in (0, 1)$ and any nonnegative continuous solution $u$ of the parabolic problem on $Q_\e = B_\e \times (0, \e^2 t_2]$.
\end{proposition}

\begin{proof}
The main goal is to show that $c$ does not depend on $\e$. Let $\e \in (0,1)$ and let $u \geq 0$ be a continuous solution of the parabolic problem on $Q_\e$.
Let us define the rescaled function $u^\e(x,t) = u(\e x, \e^2 t)$ defined on $Q_1$.
We will show that we can apply Harnack's inequality to $u^\e$.

Let $\vp$ be a smooth function and suppose that $u^\e - \vp$ has a strict maximum zero at $(x_0,t_0) \in Q_1$.
Then $u - \vp^{1/\e}$ has a strict maximum at $(\e x_0, \e^2 t_0)$, and by \eqref{eq:structural} and \eqref{eq:proper} at $(x_0,t_0)$,
\begin{align*}
\e^{-2} \pth{\vp_t - \mathcal{M}^+ (D^2 \vp) - \de_1 \abs{D\vp}} 
&\leq \e^{-2} \vp_t - \e^{-2} \mathcal{M}^+ (D^2 \vp) - \e^{-1} \de_1 \abs{D\vp}
\\&= \vp^{1/\e}_t - \mathcal{M}^+(D^2 \vp^{1/\e}) - \de_1 \abs{D\vp^{1/\e}} 
\\&\leq  \vp^{1/\e}_t - F(D^2 \vp^{1/\e}, D\vp^{1/\e}, \vp^{1/\e}) \leq 0.
\end{align*}
Therefore $u \in \underline{\mathcal{S}} (\de_1, 0, 0)$ as defined in \cite{WangI}.

Similarly, if $u - \vp$ has a strict minimum zero at $(x_0,t_0)$, then 
\begin{align*}
\e^{-2}&\pth{\vp_t - \mathcal{M}^-(D^2\vp) + \de_1 \abs{D\vp} + \de_0 \abs{\vp}} 
\\&\geq \e^{-2} \vp_t - \e^{-2}\mathcal{M}^-(D^2\vp) + \e^{-1}\de_1 \abs{D\vp} + \de_0 \abs{\vp}
\\&\geq \vp^{1/\e}_t - F(D^2 \vp^{1/\e}, D\vp^{1/\e}, \vp^{1/\e}) 
\geq 0.
\end{align*}
i.e $u \in \overline{\mathcal{S}} (\de_1, \de_0, 0)$.

Therefore we can apply the Harnack's inequality \cite{WangI}*{Theorem 4.18} to $u^\e$, since its proof relies on the weak Harnack's inequality \cite{WangI}*{Corollary 4.14} and the local maximum estimate \cite{WangI}*{Theorem 4.16}, and the function $u^\e$ satisfies the hypotheses of both.
\end{proof}

\section{Construction of barriers}
\label{ap:construction-of-barriers}

In this section we construct necessary radially symmetric barriers for \eqref{RP} used throughout the paper, first for fully nonlinear operators and then for the divergence-type operators.

\subsection{Fully nonlinear operator}
\label{ap:construction-of-barriers-fully-nonlinear}

Given a radially symmetric function $\vp$ of the form
\begin{align*}
\vp(x) = \psi(\abs{x}),
\end{align*}
its Hessian can be expressed as
\begin{align*}
D^2\vp(x) = \frac{1}{\abs{x}}\pth{I - \frac{x \otimes x}{\abs{x}^2}}\psi'(\abs{x}) + \frac{x \otimes x}{\abs{x}^2} \psi''(\abs{x}),
\end{align*}
where $(x \otimes x)_{ij} = x_ix_j$.

Note that 
\begin{align*}
(D^2 \vp(x)) (x) &= \psi''(\abs{x})x,\\
(D^2 \vp(x)) (y) &= \frac{1}{\abs{x}} \psi'(\abs{x}) y & &\text{for all } y, y \perp x
\end{align*}
Therefore the eigenvalues of $D^2 \vp(x)$ are
\begin{align}
\label{eq:radialEigenvalues}
e_1 = \ldots = e_{n-1} = \frac{1}{\abs{x}}\psi'(\abs{x}), \quad e_n = \psi''(\abs{x}),
\end{align}
and the eigenvector for $e_n$ is $\frac{x}{\abs{x}}$.
\begin{example}
Indeed, we can express the laplacian of a radially symmetric function as
\begin{align*}
\lap \vp(x) = \trace D^2\vp(x) = \frac{n -1}{\abs{x}} \psi'(\abs{x}) + \psi''(\abs{x}).
\end{align*}
\end{example}

\begin{example}
If $\psi(\rho)$ is increasing and concave then the Pucci extremal operators can be expressed as
\begin{align*}
\mathcal{M}^-(D^2\vp) &= \lambda (n-1) \frac{\vp'(\abs{x})}{\abs{x}} + \Lambda \vp''(\abs{x})\\
\mathcal{M}^+(D^2\vp) &= \Lambda (n-1) \frac{\vp'(\abs{x})}{\abs{x}} + \lambda \vp''(\abs{x}).
\end{align*}
Of course, if $\psi(\rho)$ is decreasing and convex, we simply swap $\mathcal{M}^+$ and $\mathcal{M}^-$.
\end{example}

Now we state two propositions that guarantee the existence of barriers used in the proof of the comparison theorem in section~\ref{sec:comparison}.

\begin{proposition}[Radial solutions]
\label{pr:radialSolution}
Let $F$ in \eqref{RP} satisfy \eqref{eq:structural} with given constants $\la$, $\Lambda$, $\de_1$, $\de_0$, and \eqref{eq:Fzero}. There is a constant $\rho_c = \rho_c(\la, \Lambda, \de_1, n)$ such that for any $\rho_0 \in (0, \rho_c]$, $\hat a > 0 > \hat b$, $\hat a + \hat b > 0$, and $\hat \om \geq 0$, there exists $\e > 0$ and a radially symmetric function $\hat \vp(x,t)$, decreasing in $\abs{x}$, and the following holds:
\begin{enumerate}
\romanlist
\item $\mu \hat\vp$ (resp. $-\mu \hat\vp$) is a classical subsolution (resp. supersolution) of \eqref{RP} for any $\mu > 0$ on 
\begin{align*}
K = \set{x : \rho_0 - \e < \abs{x} < \rho_0 + \e} \times (-\e, \e),
\end{align*}
\item $\abs{D\hat\vp^+} = \hat a$ and $\abs{D \hat\vp^-} = -\hat b$ at $\abs{x} = \rho_0$, $t = 0$,
\item the zero set moves with velocity $\hat\om$, i.e.
\begin{align*}
\set{x : \hat\vp(x,t) = 0} &= \set{x : \abs{x} = \rho_0 + \hat \om t} & & \text{for } t \in (-\e, \e).
\end{align*}
\end{enumerate}
\end{proposition}

\begin{proof}
To construct an increasing subsolution of the parabolic problem, we consider  the radially symmetric positive function
\begin{align*}
\vp(x,t) = ct + \psi(\rho) := c t + \al \pth{\rho^{-\ga} - \rho_0^{-\ga}} + \be \pth{\rho^2 - \rho_0^2}
\end{align*}
for some $\al, \be, c, \ga$ positive. $\al$ and $\be$ are chosen so that $\vp$ is decreasing in $\rho$. A straightforward calculation yields
\begin{align*}
\psi'' &= \al\ga(\ga+1) \rho^{-\ga - 2} + 2 \be > 0, & \psi' &= \rho\pth{-\al \ga \rho^{-\ga - 2} + 2 \be} < 0
\end{align*}

Then using \eqref{eq:radialEigenvalues} and \eqref{eq:Pucci}, we have at $\rho = \rho_0$ and $t = 0$, 
\begin{align}
\label{eq:constrSubs}
\begin{aligned}
\vp_t - F(D^2 \vp, D\vp, \vp) &\leq c - \mathcal{M}^-(D^2\vp) + \de_1\abs{D\vp}  \quad\pth{+ \de_0\abs{\vp} - F(0,0,0)} \\
 &= c - \al(\la(\ga + 1) - (n-1) \Lambda - \de_1 \rho_0) \ga \rho_0^{-\ga-2} \\
&\quad- 2 \be(\la + (n-1) \Lambda - \de_1 \rho_0)\\
& =: c - \al \tau_1 - \be \tau_2.
\end{aligned}
\end{align}
The term $\tau_2$ is positive for any $\rho_0$ that satisfy
\begin{align}
\label{eq:rho0restr}
0 < \rho_0 \leq \frac{\la + (n-1) \Lambda}{2\de_1} =: \rho_c.
\end{align}
 Furthermore, for $\rho_0$ in this range, one can choose $\ga$ large enough depending only on $\la, \Lambda, \de_1$ and $n$ so that $\tau_1$ is also positive. If $c < \be \tau_2$ as well, we observe that due to the continuity, $\vp$ with such parameters is a strict subsolution of the parabolic problem (and also the elliptic problem) in a neighborhood of $\set{(x,0): \abs{x} = \rho_0}$.

\vspace{10pt}

Given $a < 0$, $\om > 0$ and $\rho_0$ satisfying \eqref{eq:rho0restr}, we can choose parameters $\al, \be$ and $c$ in such a way that $\abs{D \vp} = \abs{a}$, $D\vp \cdot \nu = a$ and $V_\nu = \om$ at $t = 0$ and $\abs{x} = \rho_0$. Indeed, since $\vp$ is smooth, decreasing in $\rho$ and increasing in $t$, we can express the normal velocity $V_\nu$ on the zero level set $\set{\abs{x} = \rho_0}$ at $t = 0$ as
\begin{align*}
V_\nu = \frac{\vp_t}{\abs{D\vp}} = \frac{c}{2 \be \rho_0 - \al \ga \rho_0^{-\ga - 1}}.
\end{align*}
The conditions $V_\nu = \om$ and $\abs{D\vp} = \abs{a}$ yield $c = \om \abs{a}$. Then we find $\be$ large enough so that $c < \be \tau_2$ and $a < 2 \be \rho_0$ and finally we solve for $\al>0$ from $\abs{a} =\abs{D \vp} = 2 \be \rho_0 - \al \ga \rho_0^{-\ga - 1}$. Therefore \eqref{eq:constrSubs} is strictly negative on $\set{(x,0): \abs{x} = \rho_0}$, and continuity yields that $\vp$ is in fact a strict subsolution of the parabolic problem on a neighborhood of this set. In fact, due to smoothness of $\psi$, it is possible to replace the term $c t$ by the unique smooth increasing function of $t$, $\tau(t)$, with $\tau'(0) = c$, which guarantees that $\set{x: \vp(x, t) = 0} = \set{x : \abs{x} = \rho + \om t}$ for $t$ in a neighborhood of $t = 0$, while at the same time $\vp$ is still a strict supersolution. Finally, since the right-hand side in \eqref{eq:constrSubs} applied to $\mu \vp$ is 1-homogeneous in $\mu > 0$, $\mu \vp$ is also a subsolution on the same neighborhood.

\vspace{10pt}

The above construction also provides a subsolution of the elliptic problem for any $c \in \R$ and  $\be \geq 0$. Moreover, function $\tilde \vp = - \vp$ is a decreasing supersolution of the parabolic problem or a supersolution of the elliptic problem since $\mathcal{M}^-(M) = - \mathcal{M}^+(-M)$.
\end{proof}

The free boundary of a solution of \eqref{RP} can propagate arbitrary fast in some situations. In particular, if a supersolution is positive on a boundery of an open set $G$, then it will immediately become positive in $G$, as observed in the proof of Lemma~\ref{le:trunk}. We show this using a barrier constructed in the following lemma in the form of the fundamental solution for the heat equation in one dimension..

\begin{lemma}\label{le:arbitrarySpeedBarrier}
Suppose that $F$ satisfies \eqref{eq:structural} with constants $\la, \Lambda, \de_1, \de_0$. Let $G$ be a bounded open set and let $c, \de > 0$ be given positive constants. Then there exists a classical strict subsolution $\vp$ of \eqref{RP} on $\cl \Sigma$, $\Sigma := G \times (0, \de]$, such that $\vp < c$ on $\cl \Sigma$, $\vp(\cdot, 0) < 0$ on $\cl G$ and $\vp(\cdot, \de) > 0$ on $\cl G$.
\end{lemma}

\begin{proof}
We can assume that $G \neq \emptyset$. Let us define $d := \diam G$. By translating $G$ if necessary, we can also assume that $G \subset \set{x : d < x_1 < 2 d}$.  We shall consider $\phi$ of the form
\begin{align*}
\phi(x,t) &= \psi(x_1, t) - \e,
&
 \psi(x_1,t) &:= t^{-1/2} \exp \bra{-\frac{x_1^2}{4kt}}
\end{align*}
for suitable constants $k, \e > 0$ chosen below. 

First note that $\psi(\cdot, t)$ is convex for $x_1^2 > 2kt$. Then \eqref{eq:structural} and \eqref{eq:Pucci} yield
\begin{align*}
\phi_t - F(D^2 \phi, D\phi, \phi) &\leq \phi_t - \mathcal{M}^-(D^2\phi) + \de_1 \abs{D\phi} + \de_0 \abs{\phi}\\ \\
&= \al\pth{\psi_t - \la \psi_{x_1x_1} + \de_1 \abs{\psi_{x_1}} + \de_0 \abs{ \psi -\e}}\\ \\
&\leq \al \bra{\frac{(x_1^2 - 2kt)(k - \la)}{4k^2t^2} +\frac{\de_1\abs{x_1}}{2kt} + \de_0}\psi.
\end{align*}
By choosing $k$, $0<k \ll \min(\frac{d^2}{4\de}, \la)$, we can guarantee that the bracketed quantity is negative for all $t \in (0,2\de)$, $x_1 \in [d, 2d]$.

Observe that $\psi(\cdot, t)$ is decreasing when $x_1 > 0$. Moreover, for the fixed $k$, there are $\eta \in (0,\de)$ small and $\e > 0$ such that
\begin{align*}
\eta^{-1/2} \exp \bra{- \frac{d^2}{4k\eta}} < \e < (\de + \eta)^{-1/2} \exp \bra{- \frac{d^2}{k(\de + \eta)}}
\end{align*}
since the left-hand side goes to 0 and the right-hand side is bounded from below as $\eta \to 0$. With this choice of $k, \eta $ and $\e$ we see that $\phi(\cdot, \eta) <0$ in $\cl{G} \subset \set{x: d \leq x_1 \leq 2d}$ and $\phi(\cdot, \de + \eta) > 0$ on $\cl{G}$.

Now it is easy to verify that the function
\begin{align*}
\vp(x,t) = \al \phi(x, t + \eta)_+ - \frac{\al}{2} \phi(x,t+\eta)_-,
\end{align*}
where $\al > 0$ is sufficiently small so that $\vp < c$ on $\cl{\Sigma}$, has the properties asserted in the statement of the lemma.
\end{proof}

\subsection{Divergence-form operator}
\label{ap:construction-of-barriers-divergence-form}

Next we consider the operator
\begin{align}
\label{divFormProblem}
b(u)_t - \nabla \cdot \pth{\Psi(b(u)) \nabla u} = 0
\end{align}
where $b$ and $\Psi$ satisfy the assumptions from the introduction. Since this operator is not positively homogeneous of degree one in $u$ as the Pucci operators, it is necessary to construct barriers analogous to Proposition~\ref{pr:radialSolution} in two steps. In particular, the barrier of Proposition~\ref{pr:barrierDivLarge} is used in the proof of Lemma~\ref{le:ZWzeroAtContact}. We do not present a modified version of Lemma~\ref{le:arbitrarySpeedBarrier} since the barrier for problem \eqref{divFormProblem} can be constructed in a similar way.

\begin{proposition}
\label{pr:barrierDivLarge}
For given $\om \geq 0$, $\rho_0 > 0$ and $M > 0$ there exists $\eta_0 > 0$ such that for all $\eta \in (0, \eta_0)$ there is a strict classical supersolution $\vp$ of the divergence form problem \eqref{divFormProblem} on 
\begin{align*}
\Sigma := \set{(x,t) : \rho_0+\om t \leq \abs{x} \leq \rho_0 + \om t +\eta,\ t\in(-\tau, \infty)}
\end{align*}
with $\tau = \frac{\rho_0}{2\om}$ if $\om > 0$ and $\tau = \infty$ if $\om = 0$. Furthermore $\vp(x,t) = 0$ for $\abs{x} = \rho_0 + \om t$ and $\vp(x,t) > 2M$ for $\abs{x} = \rho_0 + \om t + \eta$.
\end{proposition}

\begin{proof}
Let us set
\begin{align}
\label{dfbk1k2}
k_1 &:= \max_{s\in[0,3M]}\frac{\om b'(s)}{\Psi(b(s))} + \frac{2(n-1)}{\rho_0} > 0, & k_2 &:= \max_{s \in [0,3M]} \frac{\abs{\Psi'(b(s))}b'(s)}{\Psi(b(s))}.
\end{align}
Set $\eta_0 := k_1^{-1}$. Now we can fix any $\eta \in (0, \eta_0)$, and for such $\eta$ choose $k > k_2$ large enough so that
\begin{equation}
\label{dfbk}
\frac{k - k_2}{k k_1} - \frac{1}{k} >\eta \quad\hbox{and}\quad k^{-1}\log \frac{k - k_2}{k_1} < 2M.
\end{equation}
Finally, we choose $a > 1$ large enough such that 
\begin{align*}
2M < k^{-1}\log (a k \eta + 1) < 3M.
\end{align*}
Note that this is always possible since $ \frac{k - k_2}{k k_1} - \frac{1}{k} \to \eta_0 > \eta$ and $k^{-1}\log \frac{k - k_2}{k_1} \to 0$ as $k \to \infty$, and $k \eta + 1 < \frac{k - k_2}{k_1}$.

We will show that the function $\vp$ of the form
\begin{equation}
\label{dfbpsi}
\vp(x,t) = \psi(\abs{x} - \om t - \rho_0), \qquad \psi(s) = \frac{\log (aks + 1)}{k},
\end{equation}
is the sought supersolution of \eqref{divFormProblem}.

Indeed, recall that $\abs{x} > \rho_0/2$ and $\vp(x,t) \in [0,3M)$ in $\Sigma$. With the help of $k_1$, $k_2$ defined in \eqref{dfbk1k2}, we can estimate
\begin{align*}
\frac{1}{\Psi(b(\vp))} &\pth{b(\vp)_t - \nabla \cdot \pth{\Psi(b(\vp)) \nabla \vp}}\\ \\
&=-\frac{\om b'(\vp)}{\Psi(b(\vp))} \psi' - \frac{\Psi'(b(\vp))b'(\vp)}{\Psi(b(\vp))} \psi'^2 - \frac{n-1}{\abs{x}} \psi' - \psi''\\ \\
&\geq -\psi'(k_1 + k_2 \psi') - \psi''.
\end{align*}
A straightforward differentiation of $\psi$ defined in \eqref{dfbpsi} then yields 
\begin{align*}
-\psi'(s)&(k_1 + k_2 \psi'(s)) - \psi''(s) \\ 
&= -\frac{a}{aks + 1} \pth{k_1 + k_2 \frac{a}{aks+1}} + \frac{a^2k}{(aks+1)^2}\\ 
&=\frac{a}{aks+1} \pth{\frac{a (k - k_2)}{aks+1} - k_1}>0,
\end{align*}
for all $s = \abs{x} - \om t - \rho_0 \in (0, \eta)$ due to the choice of $k$ in \eqref{dfbk} and $a > 1$.
\end{proof}

The barrier constructed in Proposition~\ref{pr:barrierDivLarge} is suitable for the construction of a barrier for the Richards problem similar to Proposition~\ref{pr:radialSolution}.

\begin{proposition}
For any $\rho_0 > 0$, $\hat a > 0 > \hat b$, $\hat a + \hat b > 0$, and $\hat \om \geq 0$ there exists $\e > 0$ and a radially symmetric function $\hat \vp(x,t)$, decreasing in $\abs{x}$, and the following holds:
\begin{enumerate}
\romanlist
\item $\hat \vp$ is a strict classical subsolution of \eqref{RP} in the divergence form \eqref{divFormProblem} on 
\begin{align*}
K = \set{x : \rho_0 - \e < \abs{x} < \rho_0 + \e} \times (-\e, \e),
\end{align*}
\item $\abs{D\hat\vp} = \hat a$ and $\abs{D\hat \vp} = - \hat b$ on $\abs{x} = \rho_0 + \om t$,
\item the zero set of $\hat\vp$ moves with velocity $\hat \om$, i.e.
\begin{align*}
\set{x: \vp(x,t) = 0} = \set{x : \abs{x} = \rho_0 + \hat \om t}.
\end{align*}
\end{enumerate}
\end{proposition}

\begin{proof}
We will use the function constructed in the proof of Proposition~\ref{pr:barrierDivLarge}. Let $\om = \hat \om$ and $M = 1$. Choose any $k > k_2$, where $k_2$ was defined in \eqref{dfbk1k2}, and set $a = \hat a$. Let $\psi(s)$ be as defined in \eqref{dfbpsi} and for some $\de > 0$ and $t \in (-\de, \de)$ let us set
\begin{align*}
\hat \vp(x,t) =
\begin{cases}
 -\psi(s) & \text{for } s \in (- \de, 0],\\  \\
 \frac{\hat b \abs{x}^{2-n}}{(2-n) (\rho_0 + \om t)^{1-n}} + s^2 & \text{for } s > 0,
 \end{cases}
\end{align*}
where $s = \abs{x} - \rho_0 - \om t$.
Then due to the computation in that proof, we can find $\de > 0$ such that 
\begin{itemize}
\item $\psi(s)$ is well-defined and $-\psi(s) \leq 1 = M$ for $s \in (-\de, 0)$,
\item $\hat \vp$ is a strict classical subsolution of \eqref{divFormProblem} in 
\begin{align*}
\set{(x,t) : \abs{x} - \rho_0 - \om t \in (-\de, 0),\ t \in (-\de, \de)},
\end{align*}
\item $\abs{D \hat \vp^+} = \hat a$, $\abs{D \hat \vp^-} = -\hat b$ on $\abs{x} = \rho_0 + \om t$.
\end{itemize}
We finish the proof by choosing $\e \in (0,\de)$ small enough so that 
\begin{align*}
K \subset \set{(x,t) : \abs{x} - \rho_0 - \om t \in (-\de , \infty),\ t\in (-\de,\de)}.
\end{align*}
\end{proof}

\medskip
\noindent\textbf{Acknowledgements.}  I. K. would like to thank Piotr Rybka for interesting discussions which prompted this research. 
N. P. would like to thank Yoshikazu Giga and the Department of Mathematics of the University of Tokyo for their hospitality. Much of the work was
done when he was a project researcher
of the University of Tokyo during 2011--2012,
supported by a grant (Kiban S, No.21224001)
from the Japan Society for the Promotion
of Science.


\begin{bibdiv}
\begin{biblist}

\bib{AL}{article}{
   author={Alt, Hans Wilhelm},
   author={Luckhaus, Stephan},
   title={Quasilinear elliptic-parabolic differential equations},
   journal={Math. Z.},
   volume={183},
   date={1983},
   number={3},
   pages={311--341},
   issn={0025-5874},
   review={\MR{706391 (85c:35059)}},
   doi={10.1007/BF01176474},
}

\bib{ATh}{article}{
   author={Armstrong, Scott},
   title={Principal half-eigenvalues of fully nonlinear homogeneous elliptic operators},
   pages={thesis}
}

\bib{BH}{article}{
   author={Bertsch, M.},
   author={Hulshof, J.},
   title={Regularity results for an elliptic-parabolic free boundary
   problem},
   journal={Trans. Amer. Math. Soc.},
   volume={297},
   date={1986},
   number={1},
   pages={337--350},
   issn={0002-9947},
   review={\MR{849483 (88c:35157)}},
   doi={10.2307/2000472},
}


\bib{BV}{article}{
   author={Br{\"a}ndle, Cristina},
   author={V{\'a}zquez, Juan Luis},
   title={Viscosity solutions for quasilinear degenerate parabolic equations
   of porous medium type},
   journal={Indiana Univ. Math. J.},
   volume={54},
   date={2005},
   number={3},
   pages={817--860},
   issn={0022-2518},
   review={\MR{2151235 (2006j:35130)}},
   doi={10.1512/iumj.2005.54.2565},
}

\bib{BW}{article}{
   author={Benilan, Philippe},
   author={Wittbold, Petra},
   title={On mild and weak solutions of elliptic-parabolic problems},
   journal={Adv. Differential Equations},
   volume={1},
   date={1996},
   number={6},
   pages={1053--1073},
   issn={1079-9389},
   review={\MR{1409899 (97e:35119)}},
}
  
\bib{CC}{book}{
   author={Caffarelli, Luis},
   author={Cabr{\'e}, Xavier},
   title={Fully nonlinear elliptic equations},
   series={American Mathematical Society Colloquium Publications},
   volume={43},
   publisher={American Mathematical Society},
   place={Providence, RI},
   date={1995},
   pages={vi+104},
   isbn={0-8218-0437-5},
   review={\MR{1351007 (96h:35046)}},
}

\bib{CS}{book}{
   author={Caffarelli, Luis},
   author={Salsa, Sandro},
   title={A geometric approach to free boundary problems},
   series={Graduate Studies in Mathematics},
   volume={68},
   publisher={American Mathematical Society},
   place={Providence, RI},
   date={2005},
   pages={x+270},
   isbn={0-8218-3784-2},
   review={\MR{2145284 (2006k:35310)}},
}
	
\bib{CV}{article}{
   author={Caffarelli, Luis},
   author={Vazquez, Juan Luis},
   title={Viscosity solutions for the porous medium equation},
   conference={
      title={Differential equations: La Pietra 1996 (Florence)},
   },
   book={
      series={Proc. Sympos. Pure Math.},
      volume={65},
      publisher={Amer. Math. Soc.},
      place={Providence, RI},
   },
   date={1999},
   pages={13--26},
   review={\MR{1662747 (99m:35029)}},
}

\bib{Carrillo99}{article}{
   author={Carrillo, Jos{\'e}},
   title={Entropy solutions for nonlinear degenerate problems},
   journal={Arch. Ration. Mech. Anal.},
   volume={147},
   date={1999},
   number={4},
   pages={269--361},
   issn={0003-9527},
   review={\MR{1709116 (2000m:35132)}},
   doi={10.1007/s002050050152},
}

\bib{CIL}{article}{
   author={Crandall, Michael G.},
   author={Ishii, Hitoshi},
   author={Lions, Pierre-Louis},
   title={User's guide to viscosity solutions of second order partial
   differential equations},
   journal={Bull. Amer. Math. Soc. (N.S.)},
   volume={27},
   date={1992},
   number={1},
   pages={1--67},
   issn={0273-0979},
   review={\MR{1118699 (92j:35050)}},
   doi={10.1090/S0273-0979-1992-00266-5},
}

\bib{CrandallLiggett}{article}{
   author={Crandall, M. G.},
   author={Liggett, T. M.},
   title={Generation of semi-groups of nonlinear transformations on general
   Banach spaces},
   journal={Amer. J. Math.},
   volume={93},
   date={1971},
   pages={265--298},
   issn={0002-9327},
   review={\MR{0287357 (44 \#4563)}},
}
\bib{DDS}{article}{
   author={Douglas, Jim, Jr.},
   author={Dupont, Todd},
   author={Serrin, James},
   title={Uniqueness and comparison theorems for nonlinear elliptic
   equations in divergence form},
   journal={Arch. Rational Mech. Anal.},
   volume={42},
   date={1971},
   pages={157--168},
   issn={0003-9527},
   review={\MR{0393829 (52 \#14637)}},
}

\bib{DiBG}{article}{
   author={DiBenedetto, Emmanuele},
   author={Gariepy, Ronald},
   title={Local behavior of solutions of an elliptic-parabolic equation},
   journal={Arch. Rational Mech. Anal.},
   volume={97},
   date={1987},
   number={1},
   pages={1--17},
   issn={0003-9527},
   review={\MR{856306 (87j:35259)}},
   doi={10.1007/BF00279843},
}

\bib{DS}{book}{
   author={Domencio, P.A.},
   author={Schwartz, F.W.},
   title={Physical and Chemical Hydrogeology},
   publisher={John Wiley and Sons, New-York},
   date={1998},
}

\bib{Evans}{book}{
   author={Evans, Lawrence C.},
   title={Partial differential equations},
   series={Graduate Studies in Mathematics},
   volume={19},
   edition={2},
   publisher={American Mathematical Society},
   place={Providence, RI},
   date={2010},
   pages={xxii+749},
   isbn={978-0-8218-4974-3},
   review={\MR{2597943 (2011c:35002)}},
}

\bib{FS}{book}{
   author={Fleming, Wendell H.},
   author={Soner, H. Mete},
   title={Controlled Markov processes and viscosity solutions},
   series={Applications of Mathematics (New York)},
   volume={25},
   publisher={Springer-Verlag},
   place={New York},
   date={1993},
   pages={xvi+428},
   isbn={0-387-97927-1},
   review={\MR{1199811 (94e:93004)}},
}

\bib{GGZ}{book}{
   author={Gajewski, Herbert},
   author={Gr{\"o}ger, Konrad},
   author={Zacharias, Klaus},
   title={Nichtlineare Operatorgleichungen und
   Operatordifferentialgleichungen},
   language={German},
   note={Mathematische Lehrb\"ucher und Monographien, II. Abteilung,
   Mathematische Monographien, Band 38},
   publisher={Akademie-Verlag},
   place={Berlin},
   date={1974},
   pages={ix+281 pp. (loose errata)},
   review={\MR{0636412 (58 \#30524a)}},
}

\bib{GT}{book}{
   author={Gilbarg, David},
   author={Trudinger, Neil S.},
   title={Elliptic partial differential equations of second order},
   note={Grundlehren der Mathematischen Wissenschaften, Vol. 224},
   publisher={Springer-Verlag},
   place={Berlin},
   date={1977},
   pages={x+401},
   isbn={3-540-08007-4},
   review={\MR{0473443 (57 \#13109)}},
}

\bib{K03}{article}{
   author={Kim, Inwon C.},
   title={Uniqueness and existence results on the Hele-Shaw and the Stefan
   problems},
   journal={Arch. Ration. Mech. Anal.},
   volume={168},
   date={2003},
   number={4},
   pages={299--328},
   issn={0003-9527},
   review={\MR{1994745 (2004k:35422)}},
   doi={10.1007/s00205-003-0251-z},
}

\bib{K02}{article}{
   author={Kim, Inwon C.},
   title={A free boundary problem arising in flame propagation},
   journal={J. Differential Equations},
   volume={191},
   date={2003},
   number={2},
   pages={470--489},
   issn={0022-0396},
   review={\MR{1978386 (2004e:35240)}},
   doi={10.1016/S0022-0396(02)00195-X},
}

\bib{K}{article}{
   author={Kim, Inwon C.},
   title={A free boundary problem with curvature},
   journal={Comm. Partial Differential Equations},
   volume={30},
   date={2005},
   number={1-3},
   pages={121--138},
   issn={0360-5302},
   review={\MR{2131048 (2006c:35305)}},
   doi={10.1081/PDE-200044474},
}

\bib{KP}{article}{
   author={Kim, Inwon C.},
   author={Po{\v{z}}{\'a}r, Norbert},
   title={Viscosity solutions for the two-phase Stefan problem},
   journal={Comm. Partial Differential Equations},
   volume={36},
   date={2011},
   number={1},
   pages={42--66},
   issn={0360-5302},
   review={\MR{2763347}},
   doi={10.1080/03605302.2010.526980},
}

\bib{KRT}{article}{
   author={Karlsen, K. H.},
   author={Risebro, N. H.},
   author={Towers, J. D.},
   title={$L^1$ stability for entropy solutions of nonlinear degenerate
   parabolic convection-diffusion equations with discontinuous coefficients},
   journal={Skr. K. Nor. Vidensk. Selsk.},
   date={2003},
   number={3},
   pages={1--49},
   issn={0368-6310},
   review={\MR{2024741 (2004j:35149)}},
}

\bib{LSU}{book}{
   author={Lady{\v{z}}enskaja, O. A.},
   author={Solonnikov, V. A.},
   author={Ural{\cprime}ceva, N. N.},
   title={Linear and quasilinear equations of parabolic type},
   language={Russian},
   series={Translated from the Russian by S. Smith. Translations of
   Mathematical Monographs, Vol. 23},
   publisher={American Mathematical Society},
   place={Providence, R.I.},
   date={1967},
   pages={xi+648},
   review={\MR{0241822 (39 \#3159b)}},
}
\bib{MR}{article}{
   author={Merz, W.},
   author={Rybka, P.},
   title={Strong solutions to the Richards equation in the unsaturated zone},
   journal={J. Math. Anal. Appl.},
   volume={371},
   date={2010},
   number={2},
   pages={741--749},
   issn={0022-247X},
   review={\MR{2670151}},
   doi={10.1016/j.jmaa.2010.05.066},
}

\bib{MV}{article}{
   author={Mannucci, Paola},
   author={Vazquez, Juan Luis},
   title={Viscosity solutions for elliptic-parabolic problems},
   journal={NoDEA Nonlinear Differential Equations Appl.},
   volume={14},
   date={2007},
   number={1-2},
   pages={75--90},
   issn={1021-9722},
   review={\MR{2346454 (2008k:35224)}},
   doi={10.1007/s00030-007-4044-1},
}
		
\bib{QS}{article}{
   author={Quaas, Alexander},
   author={Sirakov, Boyan},
   title={Principal eigenvalues and the Dirichlet problem for fully
   nonlinear elliptic operators},
   journal={Adv. Math.},
   volume={218},
   date={2008},
   number={1},
   pages={105--135},
   issn={0001-8708},
   review={\MR{2409410 (2009b:35132)}},
   doi={10.1016/j.aim.2007.12.002},
}

\bib{R}{article}{
author={Richards, L.A.},
title={Capillary conduction of liquids through porous mediums},
journal={Physics},
volume={1},
date={1931},
number={5},
pages={318--333}
}

\bib{V}{book}{
   author={V{\'a}zquez, Juan Luis},
   title={The porous medium equation},
   series={Oxford Mathematical Monographs},
   note={Mathematical theory},
   publisher={The Clarendon Press Oxford University Press},
   place={Oxford},
   date={2007},
   pages={xxii+624},
   isbn={978-0-19-856903-9},
   isbn={0-19-856903-3},
   review={\MR{2286292 (2008e:35003)}},
}

\bib{VP}{article}{
   author={van Duyn, C. J.},
   author={Peletier, L. A.},
   title={Nonstationary filtration in partially saturated porous media},
   journal={Arch. Rational Mech. Anal.},
   volume={78},
   date={1982},
   number={2},
   pages={173--198},
   issn={0003-9527},
   review={\MR{648943 (83g:76090)}},
   doi={10.1007/BF00250838},
}
		
\bib{WangI}{article}{
   author={Wang, Lihe},
   title={On the regularity theory of fully nonlinear parabolic equations.
   I},
   journal={Comm. Pure Appl. Math.},
   volume={45},
   date={1992},
   number={1},
   pages={27--76},
   issn={0010-3640},
   review={\MR{1135923 (92m:35126)}},
   doi={10.1002/cpa.3160450103},
}

\bib{W}{article}{
   author={Wang, Lihe},
   title={On the regularity theory of fully nonlinear parabolic equations.
   II},
   journal={Comm. Pure Appl. Math.},
   volume={45},
   date={1992},
   number={2},
   pages={141--178},
   issn={0010-3640},
   review={\MR{1139064 (92m:35127)}},
   doi={10.1002/cpa.3160450202},
}

\end{biblist}
\end{bibdiv}
 \end{document}